\documentclass{article}

\usepackage{latexsym}
\usepackage{amsmath}
\usepackage{amsfonts}
\usepackage{graphicx}
\usepackage{color}
\usepackage{hyperref}
\newcommand{\ds}{\displaystyle}

\newcommand{\R}{\mathbb{R}}
\newcommand{\torus}{{\mathbb{T}}}
\newcommand{\gu}{\mathbf{u}}
\newcommand{\gm}{\mathbf{m}}

\usepackage{theorem}

\theoremheaderfont{\scshape}
\newtheorem{theorem}{Theorem}[section]

\theorembodyfont{\normalfont}
\newtheorem{remark}{Remark}[section]
\newtheorem{definition}{Definition}[section]
\newenvironment{proof}{{\bf Proof }}{\hbox{~} \hfill \rule{0.5em}{0.5em}\\}
\numberwithin{equation}{section}
\begin{document}
\title{A Robin boundary condition problem of sand transport equations]
 {An existence and homogenization results of degenerate parabolic sand problem transport equations with robin boundary condition}
}


\centerline{\scshape Babou Khady Thiam
\footnote{mbaboukhthiam87.bt@gmail.com}}
\medskip
{\footnotesize
 \centerline{Universit\'e Alioune Diop de Bambey, }
  \centerline{ Equipe de recherche Analyse Non Lin\'eaire et G\'eom\'etrie  }
   \centerline{Laboratoire de Math\'ematiques de la D\'ecision et d'Analyse Num\'erique}
    \centerline{ (L.M.D.A.N) F.A.S.E.G)/F.S.T. }
}

\medskip
\centerline{\scshape M. A. M. T. Bald\'e \footnote{imouhamadouamt.balde@ucad.edu.sn}}
\medskip
{\footnotesize
 \centerline{Universit\'e Cheikh Anta Diop de Daka,  (S\'en\'egal),}
   \centerline{Laboratoire de Math\'ematiques de la D\'ecision et d'Analyse Num\'erique}
\centerline{ (L.M.D.A.N). }
} 

\medskip
\centerline{\scshape Ibrahima Faye \footnote{ibrahima.faye@uadb.edu.sn}, Diaraf Seck \footnote{diaraf.seck@ucad.edu.sn }}
\medskip
{\footnotesize
 \centerline{Universit\'e Alioune Diop de Bambey, UFR S.A.T.I.C, BP 30 Bambey (S\'en\'egal),}
   \centerline{Laboratoire de Math\'ematiques de la D\'ecision et d'Analyse Num\'erique}
\centerline{ (L.M.D.A.N). }
} 
\medskip
{\footnotesize
}



\begin{abstract}
This paper focuses on the theoretical study of
degenerate parabolic sand transport equations in a non periodic domain with Robin boundary condition.
We give existence and uniqueness results for the models which is also homogenized. Finally some corrector results are given.
\end{abstract}
Primary 35K65, 35B25, 35B10; Secondary 92F05, 86A60.
\maketitle
\section{Introduction and Results}
In this paper, we are interested to the well-posedness of models, built and studied in \cite{FaFreSe}. They are models valid for 
short, mean and long term dynamics of dunes in tidal area.
Considering the transport flow due to Van Rijn see \cite{nguetseng:1989}, Faye et al \cite{FaFreSe} showed that
\begin{equation}\label{0beq1}\left\{\begin{array}{ccc}\frac{\partial z^\epsilon}{\partial t}-\frac1\epsilon\nabla\cdot\Big(\mathcal A^\epsilon\nabla z^\epsilon\Big)=\frac1\epsilon\nabla\cdot\mathcal C^\epsilon \,\,\text{in}\,\,[0,T)\times\torus^2\\
z^\epsilon(0,x)=z_0(x)\,\,\text{in}\,\,\torus^2,\\
\end{array}\right.
\end{equation}
is a relevant model for short and mean term dynamics of dunes near the sea bed,
where $\,\,z_0\in H^1(\torus^2)$ is a given function and $\torus^2$ is the two dimensional torus. The coefficients $\mathcal A^\epsilon$ and $\mathcal C^\epsilon$ are given by
\begin{equation}\label{aceps}
\mathcal A^\epsilon(t, x)=a(1-b\epsilon \gm)g_a(|\mathbf u|),\,\,\text{and}\,\,\mathcal C^\epsilon(t,x)=
c(1-b\epsilon \gm)g_c(|\mathbf u|)\frac{\gu}{|\gu|}
\end{equation} 

where $a>0,\,\,b$ and $c$ are real numbers. 
The fields $\gu: [0, T)\times\torus^2\rightarrow \mathbb R^2$ and $\gm: [0, T)\times\torus^2\rightarrow \mathbb R$  are respectively the water velocity and the height variation due to the tide.  We suppose that they are given by 

\begin{equation}\label{epsmu1}
\gm(t,x)=\mathcal M(t,\frac t\epsilon,x)\,\,\text{and} \,\,\gu(t,x)=\mathcal U(t,\frac t\epsilon, x)
\end{equation} in the short term where $\ds \mathcal{U}$ and $\mathcal{M}$ are regular functions on 
$\mathbb{R^+}\times\mathbb{R}\times\torus^{2}$ such that\\
\begin{equation}\label{leq2}\ds\theta\longmapsto(\mathcal{U}(t,\theta,x),\mathcal{M}(t,\theta,x))\,\,\text{is periodic of period 1}\end{equation}
and 
\begin{equation}\label{coef A}
\mathcal A^\epsilon(t, x)=a(1-b\sqrt{\epsilon}\gm)g_a(|\mathbf u|),\,\,\text{and}\,\,\mathcal C^\epsilon(t,x)=c(1-b
\sqrt{ \epsilon}\gm)g_c(|\mathbf u|)\frac{\gu}{|\gu|}\end{equation}
where
\begin{equation}\label{epsmu2}
\gm(t,x)=\mathcal M(t,\frac {t}{\sqrt\epsilon},\frac t\epsilon,x)\,\,\text{and} \,\,\gu(t,x)=\mathcal U(t,\frac {t}{\sqrt\epsilon},\frac t\epsilon,x)
\end{equation} in the mean term where $\mathcal U$ and $\mathcal M$ are regular functions on 
$\mathbb{R^+}\times\mathbb{R}\times\mathbb{R}\times\torus^{2}$ such that
\begin{equation}\label{leq1}\left\{ \begin{array}{ccc}
        \ds\tau\longmapsto(\mathcal{U}(t,\tau,\theta, x),\mathcal{M}(t, \tau, \theta, x))\\
        \ds\theta\longmapsto(\mathcal{U}(t,\tau,\theta, x),\mathcal{M}(t, \tau, \theta, x))\end{array}\right.
\end{equation}
is periodic of period 1. \\
The model valid for long term dynamics of dunes is given by
\begin{equation}\label{0beq2}\left\{\begin{array}{ccc}\frac{\partial z^\epsilon}{\partial t}-\frac{1}{\epsilon^2}\nabla\cdot\Big(\mathcal A^\epsilon\nabla z^\epsilon\Big)=\frac{1}{\epsilon^2}\nabla\cdot\mathcal C^\epsilon\,\,\text{in} \,\,[0,T)\times\torus^2\\
z^\epsilon(0,x)=z_0(x)\,\,\text{in}\,\,\torus^2,\\
\end{array}\right.
\end{equation}
where $\mathcal A^\epsilon$ and $\mathcal C^\epsilon$ are given by (\ref{aceps}) and $\gu$ and $\gm$ satisfy 
\begin{equation}\label{0beq3}
\gm(t,x)=\mathcal M(t,\frac t\epsilon,x)=\mathcal M_1(\frac t\epsilon,x)+\epsilon^2 \mathcal M_2(t,\frac t\epsilon,x)
\,\,\text{and}$$
 $$\gu(t,x)=\mathcal U(t,\frac t\epsilon,x)=\mathcal U_0(\frac t\epsilon)+\epsilon\mathcal U_1(\frac t\epsilon,x)
+\epsilon^2\mathcal U_2(t,\frac t\epsilon,x).
\end{equation}
Functions $\mathcal U$ and $\mathcal M$ in (\ref{epsmu1}) having the same behavior as in (\ref{epsmu2}) satisfy the following hypotheses
\begin{equation}\label{eq40}\left\{ \begin{array}{ccc}
        \ds |\mathcal{U}|,\,\,|\frac{\partial \mathcal{U}}{\partial t}|,\,\,|\frac{\partial \mathcal{U}}{\partial \theta}|,\,\,
|\frac{\partial \mathcal{U}}{\partial \tau}|,\,\,|\nabla \mathcal{U}|,\\
|\mathcal{M}|,\,\,|\frac{\partial \mathcal{M}}{\partial t}|,\,\,|\frac{\partial \mathcal{M}}{\partial \theta}|,\,\,
|\frac{\partial \mathcal{M}}{\partial \tau}|,\,\,|\nabla \mathcal{M}| \,\,\textrm{are bounded by}\,\,d,\\
           \ds \forall (t,\tau,\theta,x)\in\mathbb{R}^{+}\times\mathbb{R}\times\mathbb{R}\times\torus^{2},\,\,|\mathcal{U}(t,\tau,\theta,x)|\leq U_{thr} \Longrightarrow
\hspace{3cm }\\ \hspace{3cm} \ds \frac{\partial\mathcal{U}}{\partial t}=0,\,\,\ds \frac{\partial\mathcal{U}}{\partial \tau}=0,
\,\, \frac{\partial\mathcal{U}}{\partial \theta}=0,\,\,
 \frac{\partial\mathcal{M}}{\partial t}=0,\,\,\ds \frac{\partial\mathcal{M}}{\partial \tau}=0,
\,\,\ds \frac{\partial\mathcal{M}}{\partial \theta}=0,\\
\,\,\nabla\mathcal{M}(t,\tau,\theta,x)=0\,\,\textrm{and}\,\,\nabla\mathcal{U}(t,\tau,\theta,x)=0,\\
            \ds\exists \theta_{\alpha}<\theta_{\omega}\in[0,1]\,\,\textrm{such that}\,\, \forall\,\,\theta\in [\theta_{\alpha},\theta_{\omega}],\,\,\text{we have}\,\,|\mathcal{U}(t,\tau,\theta,x)|\geq U_{thr}\end{array}\right.
\end{equation}
with
 $g_a$ and $g_c$ are positive functions satisfying the following hypotheses
\begin{equation}
\label{hyp1}\left\{ \begin{array}{ccc}
g_{a}\geq g_{c}\geq0,\,\,g_{c}(0)=g'_{c}(0)=0,\\~\sup_{u\in\mathbb{R}^{+}}|g_{a}(u)|+\sup_{u\in\mathbb{R}^{+}}|g'_{a}(u)|\leq d,~\\
    \sup_{u\in\mathbb{R}^{+}}|g_{c}(u)|+\sup_{u\in\mathbb{R}^{+}}|g'_{c}(u)|\leq d,\\
   \,\,\exists\, G_{thr}>0,\,\,\textrm{such that}\,\, u\geq U_{thr}\Longrightarrow g_{a}(u)\geq G_{thr}.\end{array}\right.
\end{equation}
The objective of this paper is to study the well-posedness of the models (\ref{0beq1}) and (\ref{0beq2}) coupled with (\ref{epsmu1})
or (\ref{epsmu2}) in the cases of short and mean term models or (\ref{0beq3}) in the long term model in a domain $\Omega$ with boundary $\partial\Omega$ containing the two dimensional torus $\torus^2$. 
In other words, we are interested in the existence and the uniqueness of solutions to the  following two problems posed in 
the domain  $\Omega$ of class $\mathcal C^1.$\\
\begin{equation}\label{short1}\left\{\begin{array}{ccc}\frac{\partial z^\epsilon}{\partial t}-\frac1\epsilon\nabla\cdot\Big(\mathcal A^\epsilon\nabla z^\epsilon\Big)=\frac1\epsilon\nabla\cdot\mathcal C^\epsilon\,\,\text{in} \,\,[0,T)\times\Omega\\
z^\epsilon(0,x)=z_0(x) \,\,\text{in}\,\,\Omega,\\
\frac{\partial z^\epsilon}{\partial n}+z^\epsilon= g\,\,\text{on}\,\,[0,T)\times\partial \Omega
\end{array}\right.
\end{equation}
and
\begin{equation}\label{long1}\left\{\begin{array}{ccc}\frac{\partial z^\epsilon}{\partial t}-\frac{1}{\epsilon^2}\nabla\cdot\Big(\mathcal A^\epsilon\nabla z^\epsilon\Big)=\frac{1}{\epsilon^2}\nabla\cdot\mathcal C^\epsilon\,\,\text{in} \,\,[0,T)\times\Omega\\
z^\epsilon(0,x)=z_0(x)\,\,\text{in}\,\,\Omega,\\
\frac{\partial z^\epsilon}{\partial n}+z^\epsilon= g\,\,\text{on}\,\,[0,T)\times\partial \Omega
\end{array}\right.
\end{equation}
where $g\in L^2([0,T),L^2(\Omega))$ and $z_0\in L^2(\Omega)$.\\
Then we will show that the solutions of the two problems are bounded independently of $\epsilon.$
The second step consists in the homogenization and correction results for the two problems mentioned.
Our first result is given by:

\begin{theorem} 
 \label{th1.0} Let $\Omega$ be a measurable  set of classe $\mathcal C^1.$  For all $T>0$ and $\epsilon>0$, under assumptions (\ref{eq40}), (\ref{hyp1}), (\ref{epsmu1}) and (\ref{epsmu2}) or (1.9)\, 
if $z_0\in H^1(\Omega)$ and $g\in L^2([0,T), H^1(\mathbb R^2)),$ there exists a unique function 
$z^\epsilon\in L^\infty([0,T), H^1(\Omega))$ solution to (\ref{short1}) or (\ref{long1}). This solution satisfies
\begin{equation}\label{eqborn}
\Big\| z^\epsilon\Big\|_{L^\infty([0,T), L^2(\Omega))}\leq \tilde\gamma
\end{equation}
where $\tilde\gamma$ is a constant depending only on $z_0,\,\,G_{thr}$ and $g.$\\
Moreover, this solution satisfies
\begin{equation}
\frac{d}{dt}\int_\Omega z^\epsilon(t,x)dx=0.
\end{equation}
\end{theorem}
\noindent The second one concerns homogenization results. Using the fact that (\ref{eqborn}) holds and  two scale convergence method due to Nguetseng\cite{nguetseng:1989}, Allaire \cite{allaire:1992}, Faye et al. \cite{FaFreSe}, the sequence of solutions $(z^\epsilon(t,x))$  Two-scale converges to a profile  $U$ given in the following theorem.  
\begin{theorem}
Let $\Omega$ be a measurable  set of classe $\mathcal C^1.$  For all $T>0$ and $\epsilon>0$, under assumptions (\ref{eq40}), 
(\ref{hyp1}), (\ref{epsmu1})-(\ref{leq1}) and (1.9), the sequence of solutions $(z^\epsilon(t,x))$ 
to (\ref{short1}) Two scale converges to a profile  $U\in L^{\infty}([0,T],L^{\infty}_\#(\R,L^2(\Omega)))$ solution to
\begin{equation}
\left\{\begin{array}{ccc}
\frac{\partial U}{\partial\theta}
-\nabla\cdot(\widetilde{\mathcal{A}}\nabla U)=\nabla \cdot\widetilde{\mathcal{C}}\,\,\text{in}\,\,(0,T)\times\mathbb R\times \Omega\\
\frac{\partial U}{\partial n}+U= g\,\,\text{on}\,\,(0,T)\times\mathbb R\times \partial\Omega\end{array}\right.
\end{equation}
where $\widetilde{\mathcal{A}}$ and $\widetilde{\mathcal{C}}$ are given by
\begin{equation}
\widetilde{\mathcal{A}}(t,\tau,\theta,x)=ag_a(|\mathcal U(t,\tau,\theta,x)|) \,\, \text{and}\,\,\,\,\widetilde{\mathcal{C}}(t,\tau,\theta,x)=cg_c(|\mathcal{U}(t,\tau,\theta,x)|)\,\frac{\mathcal U(t,\tau,\theta,x)}{|\mathcal{U}(t,\tau,\theta,x)|}.
\end{equation}
with $\mathcal U$ given by (\ref{leq2}) or (\ref{leq1}).\\
The sequence of solutions $(z^\epsilon(t,x))$ to (\ref{long1}) Two scale converges to a profile $U\in L^{\infty}([0,T],L^{\infty}_\#(\R,L^2(\Omega)))$ solution to 
\begin{equation}
\left\{\begin{array}{ccc}
-\nabla\cdot(\widetilde{\mathcal{A}}\nabla U)=\nabla \cdot\widetilde{\mathcal{C}}\,\,\text{in}\,\,(0,T)\times\mathbb R\times \Omega\\
\frac{\partial U}{\partial \theta}=0\,\,\text{on}\,\,\Theta_{thr}\\
\frac{\partial U}{\partial n}+U= g\,\,\text{on}\,\,(0,T)\times\mathbb R\times \partial\Omega\end{array}\right.
\end{equation} where $\widetilde{\mathcal{A}}$ and $\widetilde{\mathcal{C}}$ are given by

\begin{equation}
\widetilde{\mathcal{A}}(t,\theta,x)=ag_a(|\mathcal U_0(\theta)|) \,\, \text{and}\,\,\,\,\widetilde{\mathcal{C}}(t,\theta,x)=cg_c(|\mathcal{U}_0(\theta)|)\,\frac{\mathcal U_0(\theta)}{|\mathcal{U}_0(\theta)|}.
\end{equation}
with $\mathcal U_0$ given in (1.9).
\end{theorem}

\section{Existence and Estimates, proof of theorem \ref{th1.0}}
The objective of this paper is to study the sand transport models obtained in \cite{FaFreSe} and posed in a domain
$\Omega$ with boundary of class $\mathcal C^1$. They necessitate to impose a boundary condition in the domain.
Since the sand does not flow outside $\Omega,$ we showed in  \cite{Babou} (see especially (1.7)-(1.9) therein) that the natural boundary condition on $\partial\Omega$ is the Neumann one.
But we can also use 
a Dirichlet and mixed Neumann-Dirichlet condition. In our context, we shall consider a Robin boundary value problem. 
Consider 
\begin{equation}\label{0homoeq}
\mathcal A^\epsilon(t,x)=\widetilde{\mathcal A}_\epsilon(t,\frac t\epsilon,x),\,\,\mathcal C^\epsilon(t,x)=\widetilde{\mathcal C}_\epsilon(t,\frac t\epsilon,x)\end{equation}
where 
\begin{eqnarray}\label{homoeq}
\widetilde{\mathcal A}_\epsilon(t,\theta,x)= a(1-b\epsilon\mathcal M(t;\theta,x))g_a(|\mathcal U(t,\theta,x)|)
\\
\label{homoeq1}\widetilde{\mathcal C}_\epsilon(t,\theta,x)=c(1-b\epsilon\mathcal M(t,\theta,x))g_c(|\mathcal U(t,\theta,x)|)\frac{\mathcal U(t,\theta,x)}{|\mathcal U(t,\theta,x)|}
\end{eqnarray}
in short and long terms or
\begin{equation}\label{homoeq02}
\mathcal A^\epsilon(t,x)=\widetilde{\mathcal A}_\epsilon(t,\frac{t}{\sqrt{\epsilon}},\frac t\epsilon,x),\,\,\mathcal C^\epsilon(t,x)=\widetilde{\mathcal C}_\epsilon(t,\frac{t}{\sqrt{\epsilon}},\frac t\epsilon,x)\end{equation}
where 
\begin{eqnarray}\label{homoeq2}
\widetilde{\mathcal A}_\epsilon(t,\tau,\theta,x)= a(1-b\sqrt{\epsilon}\mathcal M(t,\tau,\theta,x))g_a(|\mathcal U(t,\tau,\theta,x)|)
\\
\label{homoeq3}\widetilde{\mathcal C}_\epsilon(t,\tau,\theta,x)=c(1-b\sqrt{\epsilon}\mathcal M(t,\tau,\theta,x))g_c(|\mathcal U(t,\tau,\theta,x)|)\frac{\mathcal U(t,\tau,\theta,x)}{|\mathcal U(t,\tau,\theta,x)|}
\end{eqnarray}
in mean term
where $\mathcal U$ and $\mathcal M$ are given in (\ref{epsmu1}) and (\ref{epsmu2}).\\ 
Under assumptions (\ref{eq40}) and (\ref{hyp1}), $\widetilde{\mathcal A}_\epsilon$ and $\widetilde{\mathcal C}_\epsilon$ satisfy the following hypotheses
\begin{equation}\label{hupp1}\left\{\begin{array}{ccc}
\tau\rightarrow (\widetilde{\mathcal A}_\epsilon,\widetilde{\mathcal C}_\epsilon)\,\,\text{is periodic of period}\,\, 1,\\
\theta\rightarrow (\widetilde{\mathcal A}_\epsilon;\widetilde{\mathcal C}_\epsilon)\,\,\text{is periodic of period}\,\,1\\
x\rightarrow (\widetilde{\mathcal A}_\epsilon,\widetilde{\mathcal C}_\epsilon)\,\,\text{is defined on } \,\,\Omega\\
|\widetilde{\mathcal A}_\epsilon|\leq \gamma,\,\, |\widetilde{\mathcal C}_\epsilon|\leq\gamma,\,\,|\frac{\partial\widetilde{\mathcal A}_\epsilon}{\partial \theta}|\leq \gamma,\,\, |\frac{\partial\widetilde{\mathcal C}_\epsilon}{\partial \theta}|\leq \gamma,\\
|\nabla\widetilde{\mathcal A}_\epsilon|\leq \epsilon^i\gamma,\,\,|\nabla\cdot\widetilde{\mathcal C}_\epsilon|\leq \epsilon^i\gamma,\\
\,\,|\frac{\partial\widetilde{\mathcal A}_\epsilon}{\partial t}|\leq \epsilon^{1+i}\gamma,\,\,|\frac{\partial\widetilde{\mathcal C}_\epsilon}{\partial t}|\leq \epsilon^{1+i}\gamma,\,\,|\frac{\partial\nabla \widetilde{\mathcal A}_\epsilon}{\partial t}|\leq\epsilon^{1+i} \gamma,\,\,|\frac{\partial\nabla\cdot\widetilde{\mathcal C}_\epsilon}{\partial t}|\leq \epsilon^{1+i} \gamma,
\end{array}\right.\end{equation}
where $i=0$ in the case of equation (\ref{short1}) and $i=1$ in the case of equation (\ref{long1}).\\
Equations (\ref{short1}) and (\ref{long1}) can be written in the generic form
\begin{equation}\label{shortlong1}\left\{\begin{array}{ccc}\frac{\partial z^\epsilon}{\partial t}-\frac{1}{\epsilon^i}\nabla\cdot\Big(\mathcal A^\epsilon\nabla z^\epsilon\Big)=\frac{1}{\epsilon^i}\nabla\cdot\mathcal C^\epsilon\,\,\text{in} \,\,[0,T)\times\Omega\\
z^\epsilon(0,x)=z_0(x)\,\,\text{in}\,\,\Omega,\\
\frac{\partial z^\epsilon}{\partial n}+z^\epsilon= g\,\,\text{on}\,\,[0,T)\times\partial \Omega
\end{array}\right.
\end{equation} for $i=1$ or $2.$ The case where $i=1$ corresponds to the valid model for short and mean term whereas
for $i=2$ we have the long term one. In the following, we are interested to the existence and uniqueness of solutions to 
(\ref{shortlong1}). This equation is a degenerated perturbed parabolic equation. The degeneracy of the coefficent 
$\mathcal A^\epsilon,$ makes that, the resolution of (\ref{shortlong1}) can not be done with the classical methods like Lax 
Milgram or Stampachia theorem. Therefore for the study of the problem, we will first prove, as in \cite{FaFreSe}, 
the existence of a periodic degenerated parabolic equation that we set out.
\\Before giving the proof of theorem \ref{th1.0}, we show that the coefficients of equation (\ref{shortlong1}), as well as their
derivatives, are bounded independently of $\epsilon.$  Existence for a given $\epsilon$ follows from results of Lions
\cite{Lions1},  Ladyzhenskaya et Solonnikov \cite{LadSol}. But, since our aim is to study the asymptotic behavior of $z^\epsilon$ 
as $\epsilon$ goes to  0, we need estimates which do not depend on $\epsilon.$ For this, based on the work of Faye et al \cite{FaFreSe},\,\cite{FaFrSe1} and \cite{Babou},
we are going to show the existence of  the following Robin boundary value problems:\\
$\forall \mu>0, \,\,\nu>0$ and $\epsilon>0,$ find  $\mathcal S^\nu=\mathcal S^\nu(t,\tau,\theta,x)$ and $\mathcal S^\nu_\mu=\mathcal S^\nu_\mu(t,\tau,\theta,x)$ periodic of period 1 in $\theta$ and solution to
\begin{equation}\label{eqreg1}\left\{\begin{array}{ccc}\frac{\partial\mathcal S^\nu}{\partial \theta}-\frac{1}{\epsilon^i}\nabla\cdot\Big(\big(\widetilde{\mathcal A}_\epsilon(t,\tau,\theta,\cdot)+\nu\big)\nabla \mathcal S^\nu\Big)=\\\frac{1}{\epsilon^i}\nabla \cdot\widetilde{\mathcal C}_\epsilon(t,\tau,\theta,\cdot) \,\,\text{in}\,\,(0,T)\times\mathbb R\times\mathbb R\times\Omega\\
 \mathcal S^\nu(0,0,0,x)=z_0(x)\,\,\text{in}\,\,\,\Omega\\
 \frac{\partial \mathcal S^\nu}{\partial n}+\mathcal S^\nu= g\,\,\,\,\text{on}\,\,(0,T)\times\mathbb R\times\mathbb R\times\partial\Omega
\end{array}\right.\end{equation}
and 
\begin{equation}\label{eqreg2}\left\{\begin{array}{ccc}\mu\mathcal S_\mu^\nu+\frac{\partial\mathcal S^\nu_\mu}{\partial \theta}-\frac{1}{\epsilon^i}\nabla\cdot\Big(\big(\widetilde{\mathcal A}_\epsilon(t,\tau,\theta,\cdot)+\nu\big)\nabla \mathcal S^\nu_\mu\Big)=\frac{1}{\epsilon^i}\nabla \cdot\widetilde{\mathcal C}_\epsilon(t,\tau,\theta,\cdot)\\ \,\,\text{in}\,\,(0,T)\times\mathbb R\times\mathbb R\times\Omega\\
 \mathcal S^\nu_\mu(0,0,0,x)=z_0(x)\,\,\text{in}\,\,\Omega\\
 \frac{\partial \mathcal S^\nu_\mu}{\partial n}+\mathcal S^\nu_\mu= g\,\,\text{on}\,\,(0,T)\times\mathbb R\times\mathbb R\times\partial\Omega,
\end{array}\right.\end{equation}
for $i=0$ or 1. 
In equations (\ref{eqreg1}) and (\ref{eqreg2}), $t$ is only a parameter. 
Let us focus on existence and uniqueness of $\mathcal S^\nu$ and $\mathcal S_\mu^\nu$ solutions to (\ref{eqreg1}) and (\ref{eqreg2}). But before proceeding further on, we point out the following remark.
\begin{remark}
We have to notice that, under assumptions (\ref{homoeq}), (\ref{eq40}) and (\ref{hyp1}), the coefficients $\widetilde{\mathcal A}_\epsilon,\,\,\widetilde{\mathcal C}_\epsilon$ and its derivatives are bounded on $\mathbb R^+\times\mathbb R\times\mathbb R\times\Omega$ by a constant 
$\epsilon^i\gamma$ where $\gamma$ not depending on $\epsilon$ for $i=1$ or 2.  
Moreover, for all $0\leq\epsilon<1,\,\,\theta\longmapsto (\widetilde{\mathcal A}_\epsilon,\,\,\widetilde{\mathcal C}_\epsilon)$ is periodic of period 1 and there exists a constant  
$\widetilde G_{thr}$ and $\theta_\alpha,\,\,\theta_\omega\in[0,1],\,\,\theta_\alpha<\theta_\omega$ such that 
\begin{equation}\label{2.6}
\widetilde{\mathcal A}_\epsilon(t,\tau,\theta,x) \geq \widetilde{G}_{thr}, \forall (t,\tau,\theta,x)\in \mathbb R^+\times\mathbb R\times\mathbb R\times\Omega\\
\end{equation}
and such that $\forall \,(t,\tau,\theta,x)\in\mathbb R^+\times\mathbb R\times\mathbb R\times\Omega$
\begin{equation}\label{2.7}
\widetilde{\mathcal{A}}_{\epsilon}(t,\tau,\theta,x)\leq\widetilde{G}_{thr}\Longrightarrow
\left\{\begin{array}{ccc}
\ds\frac{\partial\widetilde{\mathcal{A}}_{\epsilon}}{\partial t}(t,\tau,\theta,x)=0,\,\,\frac{\partial\widetilde{\mathcal{A}}_{\epsilon}}{\partial \tau}(t,\tau,\theta,x)=0,\\
\nabla\widetilde{\mathcal{A}}_{\epsilon}(t,\tau,\theta,x)=0, \vspace{3pt}
\ds\frac{\partial\widetilde{\mathcal{C}}_{\epsilon}}{\partial t}(t,\tau,\theta,x)=0,\\
\,\,\frac{\partial\widetilde{\mathcal{C}}_{\epsilon}}{\partial \tau}(t,\tau,\theta,x)=0,\,\nabla\cdot\widetilde{\mathcal{C}}_{\epsilon}(t,\tau,\theta,x)=0.\\
\end{array}\right.\end{equation}
\end{remark}
\noindent We have the following theorem:
\begin{theorem}\label{1th2} Under the same assumptions as in theorem\,\ref{th1.0} and under assumptions (\ref{hupp1}),(\ref{eq20}), (\ref{2.6}) and (\ref{2.7}),
$\forall \epsilon>0,\,\,\nu>0,\,\,\mu>0,$ there exists a unique $\mathcal S_\mu^\nu=\mathcal S_\mu^\nu(t,\tau,\theta,x)$ 1-periodic in $\theta,$ solution to (\ref{eqreg2}). Moreover there exists constants $\gamma_2,\,\,\gamma_3,\,\,\gamma_6$ which depends only on $\Omega,\,\,\gamma,\,\,\nu,\,\,\epsilon^i,\,\, g$ such that
\begin{equation}
\Big\|\mathcal S_\mu^\nu\Big\|_{L^2_\#(\mathbb R,H^1(\Omega))}\leq \tilde{\gamma},
\end{equation}
\begin{equation}
\Big\|\frac{\partial\mathcal S_\mu^\nu}{\partial\theta}\Big\|_{L_{\#}^2(\mathbb R,L^2(\Omega))}\leq \tilde{\tilde{\gamma}},
\end{equation}
\begin{equation}
 \Big\|\Delta\mathcal S^\nu_\mu\Big\|^2_{L^2_\#(\mathbb R,L^2(\Omega))}\leq \gamma_2,
\end{equation}
\begin{equation}\label{eqgradinf}
\Big\|\nabla \mathcal S_\mu^\nu\Big\|_{L_{\#}^\infty(\mathbb R,L^2(\Omega))}\leq  \gamma_3,
\end{equation}
\begin{equation}
\Big\|\mathcal S_\mu^\nu\Big\|_{L_{\#}^\infty(\mathbb R,L^2(\Omega))}\leq \gamma_3,
\end{equation}
\begin{equation}
 \Big\|\frac{\partial \mathcal S_\mu^\nu}{\partial t}\Big\|^2_{L^\infty_\#(\mathbb R,H^1(\Omega))}
\leq\gamma_6.
\end{equation}
\end{theorem}
\begin{proof}
The proof of this theorem is similar to the one of Theorem\,2.2 of Thiam, Faye and Seck \cite{Babou}. The big difference is the presence of $\frac{1}{\epsilon^i}-$ factors in (\ref{eqreg2}). Hence we only sketch the most similar arguments and focus on the management of those $\frac{1}{\epsilon^i}-$factors.\\
Multiplying (\ref{eqreg2}) by  $\mathcal S_\mu^\nu$ and integrating over $\Omega,$ we get
\begin{equation}
 \mu\int_\Omega\Big|\mathcal S_\mu^\nu\Big|^2dx+\frac12\frac{d}{d\theta}\int_{\Omega}\Big|\mathcal S_\mu^\nu\Big|^2dx+\frac{1}{\epsilon^i}\int_\Omega(\widetilde{\mathcal A}_\epsilon+\nu)\Big|\nabla\mathcal S_\mu^\nu\Big|^2dx$$
 $$-\frac{1}{\epsilon^i}\int_{\partial \Omega}(\widetilde{\mathcal A}_\epsilon+\nu)\frac{\partial S_\mu^\nu}{\partial n}\mathcal S_\mu^\nu d\sigma=
  \frac{1}{\epsilon^i}\int_\Omega\nabla\cdot\widetilde{\mathcal C}_\epsilon\,\mathcal S_\mu^\nu dx.
\end{equation}
Using robin boundary condition and Green formula's in the left hand side, we get
\begin{equation}
 \mu\int_\Omega\Big|\mathcal S_\mu^\nu\Big|^2dx+\frac12\frac{d}{d\theta}\int_{\Omega}\Big|\mathcal S_\mu^\nu\Big|^2dx+
 \frac{1}{\epsilon^i}\int_\Omega(\widetilde{\mathcal A}_\epsilon+\nu)\Big|\nabla\mathcal S_\mu^\nu\Big|^2dx$$
 $$+\frac{1}{\epsilon^i}\int_{\partial \Omega}(\widetilde{\mathcal A}_\epsilon+\nu)\Big|\mathcal S_\mu^\nu\Big|^2
 d\sigma=\frac{1}{\epsilon^i}\int_{\partial \Omega}(\widetilde{\mathcal A}_\epsilon+\nu)g
 \mathcal S_\mu^\nu d\sigma$$$$-\frac{1}{\epsilon^i}\int_\Omega\widetilde{\mathcal C}_\epsilon\cdot\nabla\mathcal S_\mu^\nu dx
 +\frac{1}{\epsilon^i}\int_{\partial\Omega}\widetilde{\mathcal C}_\epsilon\cdot n\mathcal S_\mu^\nu d\sigma
\end{equation}
or
\begin{equation}
 \mu\int_\Omega\Big|\mathcal S_\mu^\nu\Big|^2dx+\frac12\frac{d}{d\theta}\int_{\Omega}\Big|\mathcal S_\mu^\nu\Big|^2dx+
 \frac{1}{\epsilon^i}\int_\Omega(\widetilde{\mathcal A}_\epsilon+\nu)\Big|\nabla\mathcal S_\mu^\nu\Big|^2dx
 +\frac{1}{\epsilon^i}\int_{\partial \Omega}(\widetilde{\mathcal A}_\epsilon+\nu)\Big|\mathcal S_\mu^\nu\Big|^2
 d\sigma$$$$=\frac{1}{\epsilon^i}\int_{\partial \Omega}(\widetilde{\mathcal A}_\epsilon+\nu)g+
\widetilde{\mathcal C}_\epsilon\cdot n\Big)\mathcal S_\mu^\nu d\sigma-
 \frac{1}{\epsilon^i}\int_\Omega\widetilde{\mathcal C}_\epsilon\cdot\nabla\mathcal S_\mu^\nu dx
 \end{equation}
 hence
\begin{equation}\label{eqprec2}\mu \Big\|\mathcal S_\mu^\nu\Big\|^2_2+\frac12\frac{d}{d\theta}
\Big(\Big\| \mathcal S_\mu^\nu\Big\|^2_2\Big) +\frac{1}{\epsilon^i}\int_\Omega(\widetilde{\mathcal A}_\epsilon+\nu)\Big|
\nabla\mathcal S_\mu^\nu\Big|^2dx+\frac{1}{\epsilon^i}\int_{\partial \Omega}(\widetilde{\mathcal A}_\epsilon+\nu)\Big|\mathcal S_\mu^\nu\Big|^2
 d\sigma
 $$$$\leq\frac{1}{\epsilon^i}\int_{\partial \Omega}(\widetilde{\mathcal A}_\epsilon+\nu)g+
\widetilde{\mathcal C}_\epsilon\cdot n\Big)\mathcal S_\mu^\nu d\sigma+
\frac{1}{\epsilon^i}\gamma|\Omega|\Big\| \nabla\mathcal S_\mu^\nu\Big\|_{L^2(\Omega)}.\end{equation}
Integrating (\ref{eqprec2}) over $\theta \in [0, 1]$, we get
$$ \mu \Big\|\mathcal S_\mu^\nu\Big\|^2_{L^2_\#(\mathbb R,L^2(\Omega))}+\frac{1}{\epsilon^i}\int_0^1 \int_\Omega
(\widetilde{\mathcal A}_\epsilon+\nu)\Big|\nabla\mathcal S_\mu^\nu\Big|^2dx\,d\theta+
\frac{1}{\epsilon^i}\int_0^1\int_{\partial \Omega}(\widetilde{\mathcal A}_\epsilon+\nu)\Big|\mathcal S_\mu^\nu\Big|^2
d\sigma\,d\theta
$$$$\leq \frac{1}{\epsilon^i}\int_0^1\int_{\partial \Omega}(\widetilde{\mathcal A}_\epsilon+\nu)g+
\widetilde{\mathcal C}_\epsilon\cdot n\Big)\mathcal S_\mu^\nu d\sigma\,d\theta+
\frac{1}{\epsilon^i}\gamma|\Omega
|\Big\|\nabla\mathcal S_\mu^\nu\Big\|_{L^2_\#(\mathbb R,L^2(\Omega))}.$$
From this last inequality, and thanks to $\widetilde{\mathcal A}_\epsilon+\nu\geq\nu$ and because of the positivity of the first term, we get
$$ \nu\int_0^1\Big( \int_\Omega\Big|\nabla\mathcal S_\mu^\nu\Big|^2dx+
\int_{\partial \Omega}\Big|\mathcal S_\mu^\nu\Big|^2
d\sigma\Big)\,d\theta
\leq \int_0^1\int_{\partial \Omega}(\widetilde{\mathcal A}_\epsilon+\nu)g+
\widetilde{\mathcal C}_\epsilon\cdot n\Big)\mathcal S_\mu^\nu d\sigma\,d\theta$$\begin{equation}\label{ref}+
\gamma|\Omega|\Big\|\nabla\mathcal S_\mu^\nu\Big\|_{L^2_\#(\mathbb R,L^2(\Omega))}.\end{equation}
Using Friedrichs's inequality, there exists $C$ dependent only on $\Omega$ such that
$$
\int_\Omega \Big|\mathcal S_\mu^\nu\Big|^2dx\leq C\Big(\int_\Omega\Big|\nabla\mathcal S_\mu^\nu\Big|^2dx+
\int_{\partial \Omega}\Big|\mathcal S_\mu^\nu\Big|^2
d\sigma)
$$
$$
\int_\Omega\Big|\nabla\mathcal S_\mu^\nu\Big|^2dx+\int_\Omega \Big|\mathcal S_\mu^\nu\Big|^2dx\leq \Big(C+1\Big)\int_\Omega\Big|\nabla\mathcal
S_\mu^\nu\Big|^2dx+C\,\int_{\partial \Omega}\Big|\mathcal S_\mu^\nu\Big|^2d\sigma$$
\begin{equation}\label{Friedrichs}
\frac{1}{C+1}\Big(\int_\Omega\Big|\nabla\mathcal S_\mu^\nu\Big|^2dx+\int_\Omega \Big|\mathcal S_\mu^\nu\Big|^2dx\Big)\leq \int_\Omega\Big|\nabla\mathcal
S_\mu^\nu\Big|^2dx+\int_{\partial \Omega}\Big|\mathcal S_\mu^\nu\Big|^2d\sigma\end{equation}                                                                              
In the second hand, the first integral of \ref{ref} can be upper bounded as follows 
\begin{equation}\label{trace}
\int_0^1\int_{\partial \Omega}(\widetilde{\mathcal A}_\epsilon+\nu)g+
\widetilde{\mathcal C}_\epsilon\cdot n\Big)\mathcal S_\mu^\nu d\sigma\,d\theta\leq C
\Big|\partial\Omega\Big|\Big((\gamma+\nu)\Big\|g\Big\|_{L^\infty}+\gamma\Big)\Big\|\mathcal S_\mu^\nu\Big\|_{L^2_\#(\mathbb R,H^1(\Omega))}
\end{equation}
Because of (\ref{Friedrichs}) and (\ref{trace}), the inequality (\ref{ref}) rewrites as
$$ \frac{\nu}{C+1}\int_0^1\Big( \int_\Omega\Big|\nabla\mathcal S_\mu^\nu\Big|^2dx+
\int_{\Omega}\Big|\mathcal S_\mu^\nu\Big|^2dx\Big)\,d\theta
\leq C
\Big|\partial\Omega\Big|\Big((\gamma+\nu)\Big\|g\Big\|_{L^\infty}+\gamma\Big)\Big\|\mathcal S_\mu^\nu\Big\|_{L^2_\#(\mathbb R,H^1(\Omega))}
$$\begin{equation}+
\gamma|\Omega|\Big\|\nabla\mathcal S_\mu^\nu\Big\|_{L^2_\#(\mathbb R,L^2(\Omega))}.\end{equation}
or
$$ \Big\|\mathcal S_\mu^\nu\Big\|^2_{L^2_\#(\mathbb R,H^1(\Omega))}
\leq \frac{C+1}{\nu}\Big(C
\Big|\partial\Omega\Big|\Big((\gamma+\nu)\Big\|g\Big\|_{L^\infty}+\gamma\Big)\Big\|\mathcal S_\mu^\nu\Big\|_{L^2_\#(\mathbb R,H^1(\Omega))}
$$\begin{equation}\label{eqgrad}+
\gamma|\Omega|\Big\|\nabla\mathcal S_\mu^\nu\Big\|_{L^2_\#(\mathbb R,L^2(\Omega))}\Big).\end{equation}
Hence, (\ref{eqgrad}) becomes 
$$ \Big\|\mathcal S_\mu^\nu\Big\|_{L^2_\#(\mathbb R,H^1(\Omega))}
\leq \frac{C+1}{\nu}\Big(C
\Big|\partial\Omega\Big|\Big((\gamma+\nu)\Big\|g\Big\|_{L^\infty}+\gamma\Big)
+
\gamma|\Omega|\Big).$$
and finally we have:
\begin{equation}\label{eqprec3}
\Big\|\mathcal S_\mu^\nu\Big\|_{L^2_\#(\mathbb R,H^1(\Omega))}\leq\tilde{\gamma}
\end{equation}
where $\tilde{\gamma}$ is a constant depend only on $C,\,\,\nu,\,\,\gamma,\,\,g,\,\,\Omega.$\\
Multiplying (\ref{eqreg2}) by $\frac{\partial \mathcal S_\mu^\nu}{\partial \theta}$ and integrating over $\Omega$,
we get:
$$\frac12\mu\frac{d}{d\theta}(\int_{\Omega}\Big|\mathcal S_\mu^\nu\Big|^2dx)+\int_\Omega\Big|\frac{\partial \mathcal S_\mu^\nu}{\partial \theta}\Big|^2dx
+\frac{1}{\epsilon^i}\int_\Omega (\widetilde{\mathcal A}_\epsilon+\nu)\nabla\mathcal S_\mu^\nu\cdot\nabla(\frac{\partial \mathcal S_\mu^\nu}{\partial \theta})dx$$
$$
-\frac{1}{\epsilon^i}\int_{\partial\Omega} (\widetilde{\mathcal A}_\epsilon+\nu)\frac{\partial \mathcal S_\mu^\nu}{\partial n}\frac{\partial \mathcal S_\mu^\nu}{\partial \theta}d\sigma=
\frac{1}{\epsilon^i}\int_\Omega\nabla\cdot\widetilde{\mathcal C}_\epsilon \frac{\partial \mathcal S_\mu^\nu}{\partial \theta}
dx,$$
Using robin boundary condition, we get
$$\frac12\mu\frac{d}{d\theta}(\int_{\Omega}\Big|\mathcal S_\mu^\nu\Big|^2dx)+\int_\Omega\Big|\frac{\partial \mathcal S_\mu^\nu}{\partial \theta}\Big|^2dx
+\frac{1}{\epsilon^i}\int_\Omega (\widetilde{\mathcal A}_\epsilon+\nu)\nabla\mathcal S_\mu^\nu\cdot\nabla(\frac{\partial \mathcal S_\mu^\nu}{\partial \theta})dx$$
$$\frac{1}{\epsilon^i}\int_{\partial\Omega} (\widetilde{\mathcal A}_\epsilon+\nu)\mathcal S_\mu^\nu\frac{\partial \mathcal S_\mu^\nu}{\partial \theta}d\sigma
-\frac{1}{\epsilon^i}\int_{\partial\Omega} (\widetilde{\mathcal A}_\epsilon+\nu)g\frac{\partial \mathcal S_\mu^\nu}{\partial \theta}d\sigma=
\frac{1}{\epsilon^i}\int_\Omega\nabla\cdot\widetilde{\mathcal C}_\epsilon \frac{\partial \mathcal S_\mu^\nu}{\partial \theta}
dx,$$
Using the facts that 
\begin{equation}            
\int_{\Omega}\big(\widetilde{\mathcal A}_\epsilon+\nu\big)\nabla\mathcal S_\mu^\nu\cdot\nabla \frac{\partial\mathcal S_\mu^\nu}{\partial\theta}dx=\frac12\frac{d\Big(\int_\Omega(\widetilde{\mathcal A}_\epsilon+\nu)|\nabla\mathcal S_\mu^\nu|^2dx\Big)}{d\theta}-\frac12\int_\Omega\frac{\partial\widetilde{\mathcal A}_\epsilon}{\partial\theta}|\nabla\mathcal S_\mu^\nu|^2dx,
\end{equation}
and
\begin{equation}            
\int_{\partial\Omega}\big(\widetilde{\mathcal A}_\epsilon+\nu\big)\mathcal S_\mu^\nu\frac{\partial\mathcal S_\mu^\nu}{\partial\theta}d\sigma=
\frac12\frac{d\Big(\int_{\partial\Omega}(\widetilde{\mathcal A}_\epsilon+\nu)|\mathcal S_\mu^\nu|^2d\sigma\Big)}{d\theta}-\frac12\int_{\partial\Omega}\frac{\partial\widetilde{\mathcal A}_\epsilon}{\partial\theta}|\mathcal S_\mu^\nu|^2d\sigma,
\end{equation}
we get 
$$\frac12\mu\frac{d}{d\theta}(\Big\|\mathcal S_\mu^\nu\Big\|^2_2)+\Big\|\frac{\partial \mathcal S_\mu^\nu}{\partial \theta}\Big\|^2_2+\frac{1}{\epsilon^i}\frac12\frac{d\Big(\int_\Omega(\widetilde{\mathcal A}_\epsilon+\nu)|\nabla\mathcal S_\mu^\nu|^2dx
+\int_{\partial\Omega}(\widetilde{\mathcal A}_\epsilon+\nu)|\mathcal S_\mu^\nu|^2d\sigma\Big)}{d\theta}=$$$$
\frac{1}{2\epsilon^i}\Big(\int_\Omega\frac{\partial\widetilde{\mathcal A}_\epsilon}{\partial\theta}|\nabla\mathcal S_\mu^\nu|^2dx
+\int_{\partial\Omega}\frac{\partial\widetilde{\mathcal A}_\epsilon}{\partial\theta}|\mathcal S_\mu^\nu|^2d\sigma\Big)
+\frac{1}{\epsilon^i}\int_{\partial\Omega} (\widetilde{\mathcal A}_\epsilon+\nu)g\frac{\partial \mathcal S_\mu^\nu}{\partial \theta}d\sigma$$$$+\frac{1}{\epsilon^i}\int_{\partial\Omega}\widetilde{\mathcal C}_\epsilon\cdot n\frac{\partial \mathcal S_\mu^\nu}{\partial \theta}d\sigma-\frac{1}{\epsilon^i}\int_{\Omega}\widetilde{\mathcal C}_\epsilon\cdot\nabla\Big(\frac{\partial \mathcal S_\mu^\nu}{\partial \theta}\Big)dx
$$
Integrating (\ref{renouv}) over $\theta\in [0, 1]$ and using 1-periodicity of $\Big\|\mathcal S_\mu^\nu\Big\|^2$ and $\int_\Omega(\widetilde{\mathcal A}_\epsilon+\nu)|\nabla\mathcal S_\mu^\nu|^2dx$ with respect to $\theta$ we have:
$$\Big\|\frac{\partial \mathcal S_\mu^\nu}{\partial \theta}\Big\|^2_{L^2_\#(\mathbb R,L^2(\Omega))}\leq
\frac{\gamma}{2\epsilon^i}\int_0^1\Big(\int_\Omega|\nabla\mathcal S_\mu^\nu|^2dx
+\int_{\partial\Omega}|\mathcal S_\mu^\nu|^2d\sigma\Big)d\theta$$$$
+\frac{1}{\epsilon^i}\int_0^1\int_{\partial\Omega}\Big(\frac{\partial\widetilde{\mathcal A}_\epsilon}{\partial \theta}g+
\frac{\partial\widetilde{\mathcal C}_\epsilon}{\partial \theta}\cdot n\Big)\mathcal S_\mu^\nu d\sigma d\theta+
\frac{1}{\epsilon^i}\int_0^1\int_{\Omega}\frac{\partial \widetilde{\mathcal C}_\epsilon}{\partial \theta}\cdot\nabla \mathcal S_\mu^\nu dx d\theta
$$
or 
$$\Big\|\frac{\partial \mathcal S_\mu^\nu}{\partial \theta}\Big\|^2_{L^2_\#(\mathbb R,L^2(\Omega))}\leq
\frac{\gamma}{2\epsilon^i}\int_0^1\Big(\int_\Omega|\nabla\mathcal S_\mu^\nu|^2dx
+\int_{\partial\Omega}|\mathcal S_\mu^\nu|^2d\sigma\Big)d\theta$$
\begin{equation}\label{ref1}
+\frac{\gamma}{\epsilon^i}\Big(C\Big|\partial\Omega\Big|\Big(\Big\|g\Big\|_{L^\infty}+1\Big)+\Big|\Omega\Big|\Big)
\Big\|\mathcal S_\mu^\nu\Big\|_{L^2_\#(\mathbb R,H^1(\Omega))}
\end{equation}
Because of (\ref{ref}) the first integral in the second hand of (\ref{ref1}) is bounded follows as
\begin{equation}
 \int_0^1\Big(\int_\Omega|\nabla\mathcal S_\mu^\nu|^2dx
+\int_{\partial\Omega}|\mathcal S_\mu^\nu|^2d\sigma\Big)d\theta\leq
\frac{\tilde{\gamma}}{C+1}\Big\|\mathcal S_\mu^\nu\Big\|_{L^2_\#(\mathbb R,H^1(\Omega))}
\end{equation}
and (\ref{ref1}) becomes
\begin{equation}\label{ref2}
\Big\|\frac{\partial \mathcal S_\mu^\nu}{\partial \theta}\Big\|^2_{L^2_\#(\mathbb R,L^2(\Omega))}\leq
\frac{\gamma}{2\epsilon^i}\Big[\frac{\tilde{\gamma}}{C+1}+
\Big(2C\Big|\partial\Omega\Big|\Big(\Big\|g\Big\|_{L^\infty}+1\Big)+2\Big|\Omega\Big|\Big)\Big]\tilde{\gamma}
\end{equation}
(\ref{ref2}) gives finally 
$$\Big\|\frac{\partial \mathcal S_\mu^\nu}{\partial \theta}\Big\|^2_{L^2_\#(\mathbb R,L^2(\Omega))}\leq \tilde{\tilde{\gamma}}.
$$
where $\tilde{\tilde{\gamma}}$ is a constant dependant on $\tilde{\gamma}$ \\
Multiplying the equation (\ref{eqreg2}) by $-\Delta\mathcal S_\mu^\nu$ and integrating over $\Omega$, we get:
$$
\mu \int_\Omega\Big|\nabla\mathcal S^\nu_\mu\Big|^2dx-\mu \int_{\partial \Omega}\frac{\partial \mathcal S^\nu_\mu}{\partial n}\mathcal S^\nu_\mu d\sigma+
\int_\Omega\nabla\mathcal S^\nu_\mu\cdot\nabla(\frac{\partial\mathcal S^\nu_\mu}{\partial \theta})dx-
\int_{\partial \Omega}\frac{\partial \mathcal S^\nu_\mu}{\partial \theta}\frac{\partial\mathcal S^\nu_\mu}{\partial n}d\sigma
$$$$
+\frac{1}{\epsilon^i}\int_\Omega\nabla\widetilde{\mathcal A}_\epsilon\cdot\nabla\mathcal S^\nu_\mu\Delta S^\nu_\mu dx
+
\frac{1}{\epsilon^i}\int_\Omega (\widetilde{\mathcal A}_\epsilon+\nu)\Big|\Delta \mathcal S^\nu_\mu\Big|^2dx=
-\frac{1}{\epsilon^i}\int_\Omega\nabla\cdot\widetilde{\mathcal C}_\epsilon\Delta \mathcal S^\nu_\mu dx.
$$
Or
$$
\mu \int_\Omega\Big|\nabla\mathcal S^\nu_\mu\Big|^2dx+
\int_\Omega\nabla\mathcal S^\nu_\mu\cdot\nabla(\frac{\partial\mathcal S^\nu_\mu}{\partial \theta})dx
+\frac{1}{\epsilon^i}\int_\Omega\nabla\widetilde{\mathcal A}_\epsilon\cdot\nabla\mathcal S^\nu_\mu\Delta S^\nu_\mu dx
$$
$$+\frac{1}{\epsilon^i}\int_\Omega (\widetilde{\mathcal A}_\epsilon+\nu)\Big|\Delta \mathcal S^\nu_\mu\Big|^2dx=
-\frac{1}{\epsilon^i}\int_\Omega\nabla\cdot\widetilde{\mathcal C}_\epsilon\Delta \mathcal S^\nu_\mu dx+\mu \int_{\partial
\Omega}g\mathcal S^\nu_\mu d\sigma
+\int_{\partial \Omega}\frac{\partial \mathcal S^\nu_\mu}{\partial \theta}g d\sigma
$$$$
-\mu \int_{\partial\Omega}\Big|\mathcal S^\nu_\mu\Big|^2 d\sigma
-\int_{\partial \Omega}\frac{\partial \mathcal S^\nu_\mu}{\partial \theta} \mathcal S^\nu_\mu d\sigma
$$
And we deduce that
$$
\mu \Big\|\nabla\mathcal S^\nu_\mu\Big\|^2_2+\frac12\frac{d}{d\theta}(\Big\|\nabla\mathcal S^\nu_\mu\Big\|^2_2)
+\frac{1}{\epsilon^i}\int_\Omega (\widetilde{\mathcal A}_\epsilon+\nu)\Big|\Delta( \mathcal S^\nu_\mu\Big|^2dx\leq
-\frac{1}{\epsilon^i}\int_\Omega\nabla\cdot\widetilde{\mathcal C}_\epsilon\Delta \mathcal S^\nu_\mu dx$$
$$-\frac{1}{\epsilon^i}\int_\Omega\nabla\widetilde{\mathcal A}_\epsilon\cdot\nabla\mathcal S^\nu_\mu\Delta S^\nu_\mu dx
+\mu\Big|\partial\Omega\Big|\Big\| g\Big\|_{L^\infty}\Big\|\mathcal S_\mu^\nu\Big\|_{L^2(\partial \Omega)}+
\frac{d }{d\theta}\Big(\int_{\partial \Omega}g\mathcal S_\mu^\nu d\sigma\Big)
$$$$
+\mu\Big\|\mathcal S_\mu^\nu\Big\|^2_{L^2(\partial \Omega)}+
\frac12\frac{d}{d\theta}(\Big\|\mathcal S^\nu_\mu\Big\|^2_{L^2(\partial \Omega)})
$$
Using the following relation :
\begin{equation}\label{int1}
\Big| UV\Big| \leq\frac{\widetilde{\mathcal A}_\epsilon+\nu}{4} U^2+\frac{1}{\widetilde{\mathcal A}_\epsilon+\nu}V^2,
\end{equation} for $U=\Delta \mathcal S^\nu_\mu\,, V=\nabla\cdot\widetilde{\mathcal C}_\epsilon$ for the first term of the right hand side and $V=\nabla\widetilde{\mathcal A}_\epsilon\cdot\nabla\mathcal S^\nu_\mu\,,U=\Delta S^\nu_\mu$ for the second term of the right hand side
we get:
\begin{equation}\label{eqdelta}
\mu \lVert\nabla\mathcal S^\nu_\mu\rVert^2_2+\frac12\frac{d}{d\theta}(\lVert\nabla\mathcal S^\nu_\mu\rVert^2_2)
+\frac{1}{\epsilon^i}\int_\Omega (\widetilde{\mathcal A}_\epsilon+\nu)\Big|\Delta \mathcal S^\nu_\mu\Big|^2dx\leq
\frac{1}{\epsilon^i}\int_\Omega\frac{(\widetilde{\mathcal A}_\epsilon+\nu)}{4}\Big|\Delta \mathcal S^\nu_\mu\Big|^2dx$$
$$+
\frac{1}{\epsilon^i}\int_\Omega\frac{\Big|\nabla\cdot \widetilde{\mathcal C}_\epsilon\Big|^2}{\widetilde{\mathcal A}_\epsilon+\nu}dx+
\frac{1}{\epsilon^i}\int_\Omega\frac{(\widetilde{\mathcal A}_\epsilon+\nu)}{4}\Big|\Delta \mathcal S^\nu_\mu\Big|^2dx
+
\frac{1}{\epsilon^i}\int_\Omega\frac{\Big| \nabla\widetilde{\mathcal A}_\epsilon\Big|^2}{\widetilde{\mathcal A}_\epsilon+\nu}\Big|\nabla\mathcal S^\nu_\mu\Big|^2dx$$
$$+
\mu C\Big(\Big|\partial\Omega\Big|\Big\| g\Big\|_{L^\infty}\Big\|\mathcal S_\mu^\nu\Big\|_{H^1(\Omega)}+
\Big\|\mathcal S_\mu^\nu\Big\|^2_{H^1(\Omega)}\Big)+
\frac{d }{d\theta}\Big(\int_{\partial \Omega}g\mathcal S_\mu^\nu d\sigma
+\frac12\Big\|\mathcal S^\nu_\mu\Big\|^2_{L^2(\partial \Omega)}\Big)
\end{equation}
From (\ref{eqdelta})\,\,
we get
\begin{equation}\label{eqdelta21}
\mu \Big\|\nabla\mathcal S^\nu_\mu\Big\|^2_2+\frac12\frac{d}{d\theta}(\Big\|\nabla\mathcal S^\nu_\mu\Big\|^2_2)
+\frac{1}{2\epsilon^i}\int_\Omega (\widetilde{\mathcal A}_\epsilon+\nu)\Big|\Delta \mathcal S^\nu_\mu\Big|^2dx\leq
\frac{1}{\epsilon^i}\int_\Omega\frac{\Big| \nabla\widetilde{\mathcal A}_\epsilon\Big|^2}{\widetilde{\mathcal A}_\epsilon+\nu}\Big|\nabla\mathcal S^\nu_\mu\Big|^2dx$$
$$+\frac{1}{\epsilon^i}
\int_\Omega\frac{\Big|\nabla\cdot \widetilde{\mathcal C}_\epsilon\Big|^2}{\widetilde{\mathcal A}_\epsilon+\nu}dx
+\mu C\Big(\Big|\partial\Omega\Big|\Big\| g\Big\|_{L^\infty}\Big\|\mathcal S_\mu^\nu\Big\|_{H^1(\Omega)}+
\Big\|\mathcal S_\mu^\nu\Big\|^2_{H^1(\Omega)}\Big)+$$$$
\frac{d }{d\theta}\Big(\int_{\partial \Omega}g\mathcal S_\mu^\nu d\sigma
+\frac12\Big\|\mathcal S^\nu_\mu\Big\|^2_{L^2(\partial \Omega)}\Big)
\end{equation}
We have to notice that, since $\theta\rightarrow \mathcal S_\mu^\nu$ is periodic of period 1, $\theta\rightarrow |\nabla\mathcal S_\mu^\nu|^2$ and $\int_{\partial\Omega}\mathcal S^\nu_\mu d\sigma$ are also periodic of period 1, then, integrating (\ref{eqdelta21}) with respect to $\theta \in [0, 1]$
and using the fact that the first term of (\ref{eqdelta21}) in the left hand side is positive, we get:
\begin{equation}\label{bbb}
 \frac12 \nu\Big\| \Delta\mathcal S^\nu_\mu\Big\|^2_{L^2_\#(\mathbb R,L^2(\Omega))}\leq
\frac{\gamma^2}{\nu}(\Big\|\mathcal S^\nu_\mu\Big\|^2_{L^2_\#(\mathbb R,H^1(\Omega))}+|\Omega|)$$
$$
+\epsilon^i\mu C\Big(\Big|\partial\Omega\Big|\Big\| g\Big\|_{L^\infty}\Big\|\mathcal S_\mu^\nu\Big\|_{L^2_\#(\mathbb R,H^1(\Omega))}+
\Big\|\mathcal S_\mu^\nu\Big\|^2_{L^2_\#(\mathbb R,H^1(\Omega))}\Big)
\end{equation}
Thus,
\begin{equation}
\Big\| \Delta\mathcal S^\nu_\mu\Big\|^2_{L^2_\#(\mathbb R,L^2(\Omega))}\leq
\frac{2\gamma^2}{\nu^2}(\tilde{\gamma}^2+|\Omega|)
+2\nu\epsilon^i\mu C\Big(\Big|\partial\Omega\Big|\Big\| g\Big\|_{L^\infty}\tilde{\gamma}+
\tilde{\gamma}^2\Big).
\end{equation}
and then 
\begin{equation}
\Big\| \Delta\mathcal S^\nu_\mu\Big\|^2_{L^2_\#(\mathbb R,L^2(\Omega))}\leq
\gamma_2.
\end{equation}
From (\ref{eqdelta21}) we can also deduce 
\begin{equation}\label{eqdelta210}
\frac12\frac{d}{d\theta}(\Big\|\nabla\mathcal S^\nu_\mu\Big\|^2_2)
\leq
\frac{1}{\epsilon^i}\int_\Omega\frac{\Big| \nabla\widetilde{\mathcal A}_\epsilon\Big|^2}{\widetilde{\mathcal A}_\epsilon+\nu}\Big|\nabla\mathcal S^\nu_\mu\Big|^2dx+
\frac{1}{\epsilon^i}\int_\Omega\frac{\Big|\nabla\cdot \widetilde{\mathcal C}_\epsilon\Big|^2}{\widetilde{\mathcal A}_\epsilon+\nu}dx$$
$$
\mu C\Big(\Big|\partial\Omega\Big|\Big\| g\Big\|_{L^\infty}\Big\|\mathcal S_\mu^\nu\Big\|_{H^1(\Omega)}+
\Big\|\mathcal S_\mu^\nu\Big\|^2_{H^1(\Omega)}\Big)+
\frac{d }{d\theta}\Big(\int_{\partial \Omega}g\mathcal S_\mu^\nu d\sigma
+\frac12\Big\|\mathcal S^\nu_\mu\Big\|^2_{L^2(\partial \Omega)}\Big)
\end{equation}
Integrating (\ref{eqdelta210}) over $(\theta_0, \theta_1)\in[0,1]^2$, we get
\begin{equation}
\frac12\int_{\theta_0}^{\theta_1}\frac{d}{d\theta}(\Big\|\nabla\mathcal S^\nu_\mu\Big\|^2_2)d\theta\leq \frac{1}{\epsilon^i}\int_{\theta_0}^{\theta_1}
\int_\Omega\frac{\Big| \nabla\widetilde{\mathcal A}_\epsilon\Big|^2}{\widetilde{\mathcal A}_\epsilon+\nu}\Big|\nabla\mathcal S^\nu_\mu\Big|^2dx\,d\theta+
\frac{1}{\epsilon^i}\int_{\theta_0}^{\theta_1}\int_\Omega\frac{\Big|\nabla\cdot \widetilde{\mathcal C}_\epsilon\Big|^2}{\widetilde{\mathcal A}_\epsilon+\nu}dx\,d\theta
$$
$$
+\mu C\Big(\Big|\partial\Omega\Big|\Big\| g\Big\|_{L^\infty}\int_{\theta_0}^{\theta_1}\Big\|\mathcal S_\mu^\nu\Big\|_{H^1(\Omega)}d\theta+
\int_{\theta_0}^{\theta_1}\Big\|\mathcal S_\mu^\nu\Big\|^2_{H^1(\Omega)}d\theta\Big)+$$$$
\int_{\theta_0}^{\theta_1}\frac{d }{d\theta}\Big(\int_{\partial \Omega}g\mathcal S_\mu^\nu d\sigma
+\frac12\Big\|\mathcal S^\nu_\mu\Big\|^2_{L^2(\partial \Omega)}\Big)d\theta
\end{equation}
From this last inequality, we have,
\begin{equation}\Big\|\nabla\mathcal S^\nu_\mu(\theta_1,\cdot)\Big\|^2_2-\Big\|\nabla\mathcal S^\nu_\mu(\theta_0,\cdot)
\Big\|^2_2\leq\frac{2\gamma^2}{\epsilon^i\nu}\int_{\theta_0}^{\theta_1}\Big(\int_\Omega\Big|\nabla\mathcal S^\nu_\mu
\Big|^2+|\Omega|\Big)d\theta+
$$
$$
+\mu C\Big(\Big|\partial\Omega\Big|\Big\| g\Big\|_{L^\infty}\int_{\theta_0}^{\theta_1}\Big\|\mathcal S_\mu^\nu\Big\|_{H^1(\Omega)}d\theta+
\int_{\theta_0}^{\theta_1}\Big\|\mathcal S_\mu^\nu\Big\|^2_{H^1(\Omega)}d\theta\Big)+$$$$
\int_{\theta_0}^{\theta_1}\frac{d }{d\theta}\Big(\int_{\partial \Omega}g\mathcal S_\mu^\nu d\sigma
+\frac12\Big\|\mathcal S^\nu_\mu\Big\|^2_{L^2(\partial \Omega)}\Big)d\theta
$$
$$\leq \frac{2\gamma^2}{\epsilon^i\nu}\Big(\Big\|\mathcal S_\mu^\nu\Big\|^2_{L^2_\#(\mathbb R, H^1(\Omega))}+
|\Omega|\Big)
+\mu C\Big(\Big|\partial\Omega\Big|\Big\| g\Big\|_{L^\infty}\Big\|\mathcal S_\mu^\nu\Big\|_{L^2_\#(\mathbb R, H^1(\Omega))}
$$$$+\Big\|\mathcal S_\mu^\nu\Big\|^2_{L^2_\#(\mathbb R, H^1(\Omega))}\Big)
\end{equation}
we get finally 
$$
\noindent\Big\|\nabla\mathcal S^\nu_\mu(\theta_1,\cdot)\Big\|^2_2\leq
\frac{2\gamma^2}{\epsilon^i\nu}\Big(\tilde{\gamma}^2+
|\Omega|\Big)
+\tilde{\gamma}\mu C\Big(\Big|\partial\Omega\Big|\Big\| g\Big\|_{L^\infty}+\tilde{\gamma}\Big)+\tilde{\gamma}^2
$$
So, taking the supremum for all $\theta_1\in [0,1],$ we get the following inequality
$$
\Big\|\nabla\mathcal S^\nu_\mu\Big\|_{L^\infty_\#(\mathbb R,L^2(\Omega))}\leq \gamma_3,
$$ giving (\ref{eqgradinf}).\\
~
At any $\theta\in\mathbb R,$ from the Fourier expansion of $\mathcal S_\mu^\nu(t,\theta,\cdot):$

\begin{equation}
\mathcal S_\mu^\nu(t,\theta,\cdot)=\sum_{k\in\mathbb N^2}\mathcal S_k(t,\theta)e^{ik\cdot x},
\end{equation} we get that 
\begin{equation}
\|\mathcal S_\mu^\nu(t,\theta,\cdot)\|_2^2\leq\|\nabla \mathcal S_\mu^\nu(t,\theta,\cdot)\|_2^2
\end{equation} for any $\theta\in\mathbb R.$ Finally we get 

\begin{equation}\|\mathcal S_\mu^\nu\|_{L^\infty_{\#}(\mathbb R,L^2(\Omega))}\leq \gamma_3.\end{equation}
Following the same idea as in \cite{FaFreSe},  $\frac{\partial \mathcal S_\mu^\nu}{\partial t}$ is solution to
\begin{equation}\label{eqreg6}
 \mu\frac{\partial \mathcal S_\mu^\nu}{\partial t}+
 \frac{\partial}{\partial \theta}(\frac{\partial \mathcal S_\mu^\nu}{\partial t})-
 \frac{1}{\epsilon^i}\nabla\cdot((\widetilde{\mathcal A}_\epsilon+\nu)\nabla(\frac{\partial \mathcal S_\mu^\nu}{\partial t}))
 =\frac{1}{\epsilon^i}\nabla\cdot\tilde{\mathcal C^\epsilon},\,\, i=0,1,
\end{equation}
where
\begin{equation}\label{b}
\tilde{\mathcal C}^\epsilon=\frac{\partial \widetilde{\mathcal C}_\epsilon}{\partial t}+\frac{\partial \widetilde{\mathcal A}_\epsilon}{\partial t}\nabla\mathcal S_\mu^\nu
\end{equation}
and
\begin{equation}\label{ba}
\nabla\cdot\tilde{\mathcal C^\epsilon}=\frac{\partial(\nabla\cdot\widetilde{\mathcal C}_\epsilon)}{\partial t}+
\frac{\partial (\nabla\widetilde{\mathcal A}_\epsilon)}{\partial t}\cdot\nabla S_\mu^\nu+
\frac{\partial \widetilde{\mathcal A}_\epsilon}{\partial t}\Delta S_\mu^\nu
\end{equation}
From relations (\ref{b}) and (\ref{ba}) and the above estimates, we deduce that their exits a constant
$\gamma_4$ depending only on $\nu$, $\tilde{\gamma},\,\,\Omega$ and $\gamma_2$ such that\\
 \begin{equation}\left\{\begin{array}{ccc}
\Big\| \tilde{\mathcal C^\epsilon} \Big\|_{L^2_\#(\mathbb R,L^2(\Omega))}\leq \epsilon^i \gamma_4  \\
\Big\| \nabla\cdot\tilde{\mathcal C^\epsilon} \Big\|_{L^2_\#(\mathbb R,L^2(\Omega))}\leq\epsilon^i \gamma_4,\,\,i=1,2
\end{array}\right.\end{equation}
with
$$
\gamma_4=\gamma\Big(1+\tilde{\gamma}+\sqrt{\gamma_2}\Big)
$$
Multiplying (\ref{eqreg6}) by $\frac{\partial S_\mu^\nu}{\partial t}$ and integrating over $\Omega,$ we have 

$$
\mu\int_\Omega \Big|\frac{\partial \mathcal S_\mu^\nu}{\partial t}\Big|^2+\frac12\frac{d}{d\theta}(\int_\Omega \big|\frac{\partial \mathcal S_\mu^\nu}{\partial t}\Big|^2dx)
+\frac{1}{\epsilon^i}\int_\Omega (\widetilde{\mathcal A}_\epsilon+\nu)\Big|\nabla(\frac{\partial \mathcal S_\mu^\nu}{\partial t})\Big|^2dx$$
$$-\frac{1}{\epsilon^i}
\int_{\partial\Omega} (\widetilde{\mathcal A}_\epsilon+\nu)\frac{\partial(\frac{\partial \mathcal S_\mu^\nu}{\partial t})}{\partial n}\frac{\partial \mathcal S_\mu^\nu}{\partial t}d\sigma=
\frac{1}{\epsilon^i}\int_\Omega \nabla\cdot\tilde{\mathcal C^\epsilon}\frac{\partial \mathcal S_\mu^\nu}{\partial t}dx.
$$
This is equivalent to
\begin{equation}\label{qqdr}
\mu\Big\|\frac{\partial \mathcal S_\mu^\nu}{\partial t}\Big\|^2_2+\frac12\frac{d}{d\theta}( \Big\|\frac{\partial \mathcal S_\mu^\nu}{\partial t}\Big\|^2_2)
+\frac{1}{\epsilon^i}\int_\Omega (\widetilde{\mathcal A}_\epsilon+\nu)\Big|\nabla(\frac{\partial \mathcal S_\mu^\nu}{\partial t})\Big|^2dx
+\frac{1}{\epsilon^i}\int_{\partial\Omega} (\widetilde{\mathcal A}_\epsilon+\nu)\Big|\frac{\partial \mathcal S_\mu^\nu}{\partial t}\Big|^2d\sigma
$$
$$=
\frac{1}{\epsilon^i}\int_{\partial\Omega} \Big((\widetilde{\mathcal A}_\epsilon+\nu)\frac{\partial g}{\partial t}+\widetilde{\mathcal C}^\epsilon\cdot n\Big)\frac{\partial \mathcal S_\mu^\nu}{\partial t}d\sigma
-\frac{1}{\epsilon^i}\int_\Omega \tilde{\mathcal C^\epsilon}\nabla\frac{\partial \mathcal S_\mu^\nu}{\partial t}dx
\end{equation}
Using (\ref{b}), we have 
\begin{equation}\label{gt1}
\int_{\partial\Omega} \Big((\widetilde{\mathcal A}_\epsilon+\nu)\frac{\partial g}{\partial t}+ \widetilde{\mathcal C}^\epsilon\cdot n\Big) \frac{\partial \mathcal S_\mu^\nu}{\partial t}d\sigma
\leq C\Big((\gamma+\nu) \Big\|\frac{\partial g}{\partial t}\Big\|_{L^\infty}+\gamma_4\Big)\Big|\partial\Omega\Big|
\Big\|\frac{\partial \mathcal S_\mu^\nu}{\partial t}\Big\|_{H^1(\Omega)}
\end{equation}
then integrating (\ref{qqdr}) from 0 to 1 with respect to $\theta,$ we have
\begin{equation}\label{grandcor}\mu\Big\|\frac{\partial \mathcal S_\mu^\nu}{\partial t}\Big\|^2_{L^2_\#(\mathbb R, L^2(\Omega))}
+\frac{1}{\epsilon^i}\int_0^1\int_\Omega\big(\widetilde{\mathcal A}_\epsilon+\nu\Big)\Big|\nabla\frac{\partial \mathcal S_\mu^\nu}{\partial t}\Big|^2dxd\theta
+\frac{1}{\epsilon^i}\int_0^1\int_{\partial\Omega}\big(\widetilde{\mathcal A}_\epsilon+\nu\Big)\Big|\frac{\partial \mathcal S_\mu^\nu}{\partial t}\Big|^2d\sigma d\theta
$$$$
\leq\frac{\gamma_4}{\epsilon^i}\int_0^1\Big\|\nabla\frac{\partial \mathcal S_\mu^\nu}{\partial t} \Big\|_{L^2(\Omega)}d\theta
+\frac{C}{\epsilon^i} \Big((\gamma+\nu) \Big\|\frac{\partial g}{\partial t}\Big\|_{L^\infty}+\gamma_4\Big)\Big|\partial\Omega\Big|
\Big\|\frac{\partial \mathcal S_\mu^\nu}{\partial t}\Big\|_{L^2_\#(\mathbb R, H^1(\Omega))}
\end{equation}
Using Friedrichs'inequality, (\ref{grandcor}) becomes
\begin{equation}
\frac{\nu}{C+1}\Big\|\frac{\partial \mathcal S_\mu^\nu}{\partial t}\Big\|^2_{L^2_\#(\mathbb R, H^1(\Omega))}
\leq\gamma_4\Big\|\frac{\partial \mathcal S_\mu^\nu}{\partial t} \Big\|_{L^2_\#(\mathbb R, H^1(\Omega))}
+$$$$C \Big((\gamma+\nu) \Big\|\frac{\partial g}{\partial t}\Big\|_{L^\infty}+\gamma_4\Big)\Big|\partial\Omega\Big|
\Big\|\frac{\partial \mathcal S_\mu^\nu}{\partial t}\Big\|_{L^2_\#(\mathbb R, H^1(\Omega))}
\end{equation}
Hence, we have
\begin{equation}\label{gamma4} 
\Big\|\frac{\partial \mathcal S_\mu^\nu}{\partial t}\Big\|_{L^2_\#(\mathbb R, H^1(\Omega))}
\leq\frac{C+1}{\nu}\Big[\gamma_4
+C \Big((\gamma+\nu) \Big\|\frac{\partial g}{\partial t}\Big\|_{L^\infty}+\gamma_4\Big)\Big|\partial\Omega\Big|\Big]
\end{equation}
From (\ref{gamma4}), there exists  $\theta_0\in[0, 1]$ such that
\begin{equation}\label{gamma3}
\Big\|\frac{\partial \mathcal S_\mu^\nu(\theta_0,\cdot)}{\partial t}\Big\|_{H^1(\Omega)}\leq \frac{\gamma_5}{\nu}.
\end{equation}
where $\gamma_5$ is a constant depending only on $\gamma_4,\,\,g,\,\,\Omega,\,\,C$\\
From (\ref{qqdr}), we get 
\begin{equation}\label{preced}
\frac12\frac{d}{d\theta}( \Big\|\frac{\partial \mathcal S_\mu^\nu}{\partial t}\Big\|^2_2\leq
\frac{\gamma_4}{\epsilon^i}\Big\|\nabla\frac{\partial \mathcal S_\mu^\nu}{\partial t}\Big\|_2
+\frac{C}{\epsilon_i}\Big((\gamma+\nu) \Big\|\frac{\partial g}{\partial t}\Big\|_{L^\infty}+\gamma_4\Big)\Big|\partial\Omega\Big|
\Big\|\frac{\partial \mathcal S_\mu^\nu}{\partial t}\Big\|_{H^1(\Omega)}.
\end{equation}
Integrating (\ref{preced}) from   $\theta_0$ given in (\ref{gamma3}) to $\theta_1 \in[0, 1]$, we obtain
\begin{equation}
\Big\|\frac{\partial \mathcal S_\mu^\nu(\theta_1, \cdot)}{\partial t}\Big\|^2_2\leq
\frac{2}{\epsilon_i}\Big(\gamma_4
+C\Big((\gamma+\nu) \Big\|\frac{\partial g}{\partial t}\Big\|_{L^\infty}+\gamma_4\Big)
\Big|\partial\Omega\Big|\Big)
\Big\|\frac{\partial \mathcal S_\mu^\nu}{\partial t}\Big\|_{L^2_\#(\mathbb R,H^1(\Omega))}$$$$+
\Big\|\frac{\partial \mathcal S_\mu^\nu(\theta_0,\cdot)}{\partial t}\Big\|^2_2.
\end{equation}
Then we have,
$$
\Big\|\frac{\partial \mathcal S_\mu^\nu}{\partial t}\Big\|^2_{L^\infty_\#(\mathbb R,L^2(\Omega))}
\leq\Big[
\frac{2}{\epsilon_i}\Big(\gamma_4
+C\Big((\gamma+\nu) \Big\|\frac{\partial g}{\partial t}\Big\|_{L^\infty}+\gamma_4\Big)
\Big|\partial\Omega\Big|\Big)+\frac{\gamma_5}{\nu}\Big]\frac{\gamma_5}{\nu}
$$
and then
$$
\Big\|\frac{\partial \mathcal S_\mu^\nu}{\partial t}\Big\|^2_{L^\infty_\#(\mathbb R,{L^2(\Omega)})}
\leq \gamma_6.
$$
\end{proof}
\noindent Because of theorem\,\ref{th2}, the sequence of solutions $(\mathcal S_\mu^\nu)_\mu$ to (\ref{eqreg2}) is bounded 
in $L^2_\#(\mathbb R,H^1(\Omega)),$ which is a reflexive space, then there exists a sub-sequence still denoted by  $(\mathcal S_\mu^\nu)_\mu$ and  a profile $(\mathcal S^\nu)$ such that
$(\mathcal S_\mu^\nu)_\mu\rightharpoonup \mathcal S^\nu\,\,\text{in}\,\,L^2_\#(\mathbb R,L^2(\Omega)),\,\,$\\$(\nabla\mathcal S_\mu^\nu)_\mu\rightharpoonup \nabla\mathcal S^\nu\,\,\text{in}\,\,L^2_\#(\mathbb R,L^2(\Omega))$ and $(\Delta \mathcal S_\mu^\nu)_\mu\rightharpoonup \Delta \mathcal S^\nu\,\,\text{in}\,\,L^2_\#(\mathbb R,L^2(\Omega)).$\\ 
Because of the $\theta$ dependance of 
$\mathcal S_\mu^\nu,$ we have also $\frac{\partial\mathcal S^\nu_\mu}{\partial \theta}\rightharpoonup \frac{\partial\mathcal S^\nu}{\partial \theta}\,\,\text{in}\,\,L^2_\#(\mathbb R,L^2(\Omega)).$\\
Then, multiplying (\ref{eqreg2}) by  a test function $v\in H^1(\Omega)$ and integrating over $\mathbb R\times\Omega$, we get
\begin{equation}
\mu\int_0^1\int_\Omega\mathcal S_\mu^\nu \,vdxd\theta+\int_0^1\int_\Omega\frac{\partial\mathcal S^\nu_\mu}{\partial \theta}\,vdxd\theta+
\frac{1}{\epsilon^i}\int_0^1\int_\Omega\big(\widetilde{\mathcal A}_\epsilon+\nu\big)\nabla \mathcal S^\nu_\mu\nabla vdxd\theta 
+$$$$\frac{1}{\epsilon^i}\int_0^1\int_{\partial\Omega}\big(\widetilde{\mathcal A}_\epsilon+\nu\big)\mathcal S^\nu_\mu vd\sigma d\theta
=\frac{1}{\epsilon^i}\int_0^1\int_\Omega\nabla \cdot\widetilde{\mathcal C}_\epsilon v \,dxd\theta+\frac{1}{\epsilon^i}\int_0^1\int_{\partial\Omega}\big(\widetilde{\mathcal A}_\epsilon+\nu\big)gvd\sigma d\theta.
\end{equation}
Passing to the limit as $\mu\rightarrow0$ we get
\begin{equation}\label{bord1}
\int_0^1\int_\Omega\frac{\partial\mathcal S^\nu}{\partial \theta}\,vdxd\theta+
\frac{1}{\epsilon^i}\int_0^1\int_{\partial\Omega}\big(\widetilde{\mathcal A}_\epsilon+\nu\big)\mathcal S^\nu vd\sigma d\theta+
\frac{1}{\epsilon^i}\int_0^1\int_\Omega\big(\widetilde{\mathcal A}_\epsilon+\nu\big)\nabla \mathcal S^\nu\nabla vdxd\theta=$$
$$\frac{1}{\epsilon^i}\int_0^1\int_\Omega\nabla \cdot\widetilde{\mathcal C}_\epsilon v \,dxd\theta
+\frac{1}{\epsilon^i}\int_0^1\int_{\partial\Omega}\big(\widetilde{\mathcal A}_\epsilon+\nu\big)gvd\sigma d\theta,\,\,\forall v\in H^1(\Omega).
\end{equation}
Taking $v\in H^1(\Omega)$ such that $v=0\,\,\text{on}\,\,\partial\Omega,$ we get
\begin{equation}\label{bord2}
\int_0^1\int_\Omega\frac{\partial\mathcal S^\nu}{\partial \theta}\,vdxd\theta+\frac{1}{\epsilon^i}\int_0^1\int_\Omega\big(\widetilde{\mathcal A}_\epsilon+\nu\big)\nabla \mathcal S^\nu\nabla vdxd\theta=\frac{1}{\epsilon^i}\int_0^1\int_\Omega\nabla \cdot\widetilde{\mathcal C}_\epsilon v \,dxd\theta.
\end{equation}
Using Green's formula, in the second term with respect to the variable $x$, we get, in the distributions sense,
\begin{equation}
\frac{\partial\mathcal S^\nu}{\partial \theta}-\frac{1}{\epsilon^i}\nabla\cdot\Big(\Big(\widetilde{\mathcal A}_\epsilon+\nu\big)\nabla\mathcal S^\nu\Big)=\frac{1}{\epsilon^i}\nabla\cdot\widetilde{\mathcal C}_\epsilon, \,\,i=0,1.
\end{equation}
Using again Green's formula in the second term of (\ref{bord1}) and using (\ref{bord2}), we get 
\begin{equation}
\forall \,\, v\in H^1(\Omega)$$$$\int_0^1\int_{\partial\Omega}\big(\widetilde{\mathcal A}_\epsilon+\nu\big)\frac{\partial \mathcal S^\nu}{\partial n}vd\sigma\,d\theta
+\int_0^1\int_{\partial\Omega}\big(\widetilde{\mathcal A}_\epsilon+\nu\big)\mathcal S^\nu vd\sigma d\theta
=\int_0^1\int_{\partial\Omega}\big(\widetilde{\mathcal A}_\epsilon+\nu\big)g v\,d\sigma d\theta
\end{equation} giving
\begin{equation}\frac{\partial \mathcal S^\nu}{\partial n}+
 \mathcal S^\nu
=g\,\,\text{on}\,\,\mathbb R\times\partial\Omega.
\end{equation}

\begin{theorem} Under the same assumptions as in theorem \ref{1th2},
$\forall \epsilon>0,\,\,\nu>0,\,\,$ there exists a unique $\mathcal S^\nu=\mathcal S^\nu(t,\theta,x)$ periodic of period 1 solution to (\ref{eqreg1}). Moreover the following inequalities holds
\begin{equation}
\Big\|\mathcal S^\nu\Big\|_{L^2_\#(\mathbb R,H^1(\Omega))}\leq \gamma_7,
\end{equation}
\begin{equation}
\Big\|\frac{\partial\mathcal S^\nu}{\partial\theta}\Big\|^2_{L_{\#}^2(\mathbb R,L^2(\Omega))}\leq \tilde{\tilde{\tilde\gamma}}
\end{equation}
\begin{equation}\label{gradnuinf}
\Big\|\Delta \mathcal S^\nu\Big\|^2_{L_{\#}^2(\mathbb R,L^2(\Omega))}\leq\frac{\gamma^2}{\tilde{G}_{thr}^2}
(\gamma_7^2+1),
\end{equation}
\begin{equation}\label{bib}
 \Big\|\frac{\partial \mathcal S^\nu(\theta,\cdot)}{\partial t}\Big\|^2_{L^2_\#(\mathbb R,H^1(\Omega))}
\leq\tilde{\gamma}_8.
\end{equation}
\end{theorem}
The idea in this part is to bound $\mathcal S^\nu$ and its derivative in a appropriate space by finite constants if 
$\nu$ tends to 0.\\
\begin{proof} 
To show that solution (\ref{eqreg1}) is unique, let $\mathcal S^\nu$ and
$\bar{\mathcal S^\nu}$ be two solutions of  (\ref{eqreg1}),  then $\mathcal S^\nu-\bar{\mathcal S^\nu}$ is also solution to:
%
\begin{equation}\label{babou}\left\{\begin{array}{ccc}
\frac{\partial(\mathcal S^\nu-\bar{\mathcal S^\nu})}{\partial \theta}-
\frac{1}{\epsilon^i}\nabla\cdot\Big(\big(\widetilde{\mathcal A}_\epsilon+\nu\big)\nabla (\mathcal S^\nu-\bar{\mathcal S^\nu})\Big)=0 \,\,\text{in}\,\,]0,T]\times\mathbb R\times\Omega
\\
\mathcal S^\nu(0,0,x)-\bar{\mathcal S^\nu}(0,0,x)=0\,\,\text{in}\,\,\,\Omega\\
\frac{\partial (\mathcal S^\nu-\bar{\mathcal S^\nu})}{\partial n}
+\mathcal S^\nu-\bar{\mathcal S^\nu}
= 0\,\,\text{on}\,\,\,]0,T]\times\mathbb R\times\partial\Omega.
\end{array}\right.
\end{equation}
Multiplying (\ref{babou}) by $\mathcal S^\nu-\bar{\mathcal S^\nu}$ and integrating over $\Omega$,
we get:
%
%
$$\frac12\frac{d}{d\theta}\int_{\Omega}\Big|(\mathcal S^\nu-\bar{\mathcal S^\nu})\Big|^2dx+
\frac{1}{\epsilon^i}\int_\Omega(\widetilde{\mathcal A}_\epsilon+\nu)\Big|\nabla(\mathcal S^\nu-\bar{\mathcal S^\nu})\Big|^2dx$$
$$-
\frac{1}{\epsilon^i}\int_{\Omega}(\widetilde{\mathcal A}_\epsilon+\nu)\frac{\partial (\mathcal S^\nu-\bar{\mathcal S^\nu})}{\partial n}(\mathcal S^\nu-\bar{\mathcal S^\nu}) d\sigma
=0.$$
Since  $\frac{\partial (\mathcal S^\nu-\bar{\mathcal S^\nu})}{\partial n}= 0\quad\text{on}\,\,\,]0,T]\times\mathbb R\times\partial\Omega$ there is no boundary term, hence we get
\begin{equation}\label{babou01}\frac12\frac{d}{d\theta}\Big(\Big\|(\mathcal S^\nu-\bar{\mathcal S^\nu})\Big\|^2_2\Big)+
\frac{\nu}{\epsilon^i}\Big\|\nabla(\mathcal S^\nu-\bar{\mathcal S^\nu})\Big\|^2_2=0
\end{equation}
which gives, as the second term of (\ref{babou01}) is positive
\begin{equation}\label{babou2}
\frac12\frac{d}{d\theta}\Big(\Big\|(\mathcal S^\nu-\bar{\mathcal S^\nu})\Big\|^2_2\Big)\leq0\end{equation}
Integrating (\ref{babou2}) from 0 to 1 and using periodicity of $\mathcal S^\nu $ we get from (\ref{babou01}) the following inequality 
$$
\nu\int_0^1\Big\|\nabla(\mathcal S^\nu-\bar{\mathcal S^\nu})\Big\|^2_2d\theta\leq0.
$$
Thus, 
$$
\nu\Big\|\nabla(\mathcal S^\nu-\bar{\mathcal S^\nu})\Big\|^2_{L^2_\#(\mathbb R,L^2(\Omega))}=0.
$$
$\forall t\in[0,T),\,\,\forall \theta\in\mathbb R,\,\,\text{then}\,\,\mathcal S^\nu(\theta)-\bar{\mathcal S}^\nu(\theta)\in H_0^1(\Omega)$
then using Poincare's inequality, there exists a constant $c(\Omega)$ such that
$$
 \mathcal S^\nu-\bar{\mathcal S^\nu} \in H^1_0(\Omega), \quad
\Big\|\mathcal S^\nu-\bar{\mathcal S^\nu}\Big\|^2_{L^2_\#(\mathbb R,L^2(\Omega))}\leq
c(\Omega)\Big\|\nabla(\mathcal S^\nu-\bar{\mathcal S^\nu})\Big\|^2_{L^2_\#(\mathbb R,L^2(\Omega))}
$$
Thus, we have
$$
\Big\|\mathcal S^\nu-\bar{\mathcal S^\nu}\Big\|^2_{L^2_\#(\mathbb R,L^2(\Omega))}=0,
$$
giving
$$
 S^\nu=\bar{\mathcal S^\nu}.
$$
Multiplying the equation (\ref{eqreg1}) by  $\mathcal S^\nu$ and integrating over $\Omega,$ we get
\begin{equation}
\frac12\frac{d}{d\theta}\int_{\Omega}\Big|\mathcal S^\nu\Big|^2dx+\frac{1}{\epsilon^i}\int_\Omega(\widetilde{\mathcal A}_\epsilon+\nu)\Big|\nabla\mathcal S^\nu\Big|^2dx-\frac{1}{\epsilon^i}\int_{\partial \Omega}(\widetilde{\mathcal A}_\epsilon+\nu)\frac{\partial\mathcal S^\nu}{\partial n}\mathcal S^\nu d\sigma=$$
$$ \frac{1}{\epsilon^i}\int_\Omega\nabla\cdot\widetilde{\mathcal C}_\epsilon\,\mathcal S^\nu dx.
\end{equation}
Using the Green formula in the right hand side of the following equality we get
\begin{equation}\label{2.73}
\frac12\frac{d}{d\theta}\int_{\Omega}\Big|\mathcal S^\nu\Big|^2dx+\frac{1}{\epsilon^i}\int_\Omega(\widetilde{\mathcal A}_\epsilon+\nu)\Big|\nabla\mathcal S^\nu\Big|^2dx-\frac{1}{\epsilon^i}\int_{\partial \Omega}(\widetilde{\mathcal A}_\epsilon+\nu)\frac{\partial \mathcal S^\nu}{\partial n}\mathcal S^\nu d\sigma= $$
$$-\frac{1}{\epsilon^i}\int_\Omega\widetilde{\mathcal C}_\epsilon\cdot\nabla\mathcal S^\nu dx+\frac{1}{\epsilon^i}\int_{\partial\Omega}\widetilde{\mathcal C}_\epsilon\cdot n\,\mathcal S^\nu\, d\sigma.
\end{equation}
Using the robin boundary condition, we have 
\begin{equation}\label{babou00}
\frac12\frac{d}{d\theta}\int_{\Omega}\Big|\mathcal S^\nu\Big|^2dx+
\frac{1}{\epsilon^i}\int_{\partial\Omega}(\widetilde{\mathcal A}_\epsilon+\nu)\Big|\mathcal S^\nu\Big|^2d\sigma
+\frac{1}{\epsilon^i}\int_\Omega(\widetilde{\mathcal A}_\epsilon+\nu)\Big|\nabla\mathcal S^\nu\Big|^2dx=$$
$$\frac{1}{\epsilon^i}\int_{\partial \Omega}\Big((\widetilde{\mathcal A}_\epsilon+\nu)g+ \widetilde{\mathcal C}_\epsilon\cdot n\Big)\mathcal S^\nu \,d\sigma -\frac{1}{\epsilon^i}\int_\Omega\widetilde{\mathcal C}_\epsilon\cdot\nabla\mathcal S^\nu dx
\end{equation}
Then, integrating (\ref{babou00}) from 0 to 1 and using periodicity of $\mathcal S^\nu$  we get
\begin{equation}\label{babou001}
\frac{1}{\epsilon^i}\int_0^1\int_{\partial\Omega}(\widetilde{\mathcal A}_\epsilon+\nu)\Big|\mathcal S^\nu\Big|^2d\sigma
+\frac{1}{\epsilon^i}\int_0^1\int_\Omega(\widetilde{\mathcal A}_\epsilon+\nu)\Big|\nabla\mathcal S^\nu\Big|^2dx=$$
$$\frac{1}{\epsilon^i}\int_0^1\int_{\partial \Omega}\Big((\widetilde{\mathcal A}_\epsilon+\nu)g+ 
\widetilde{\mathcal C}_\epsilon\cdot n\Big)\mathcal S^\nu \,d\sigma -\frac{1}{\epsilon^i}
\int_0^1\int_\Omega\widetilde{\mathcal C}_\epsilon\cdot\nabla\mathcal S^\nu dx
\end{equation}
From (\ref{babou001}), since $\widetilde{\mathcal A}_\epsilon+\nu\geq\widetilde{\mathcal A}_\epsilon$ we deduce
\begin{equation}\label{babou002}
\int_0^1\int_{\partial\Omega}\widetilde{\mathcal A}_\epsilon\Big|\mathcal S^\nu\Big|^2d\sigma
+\int_0^1\int_\Omega\widetilde{\mathcal A}_\epsilon\Big|\nabla\mathcal S^\nu\Big|^2dx\leq
\Big|\partial \Omega\Big|\Big((\gamma+\nu)\Big\|g\Big\|_{L^\infty_\#(\R,L^2(\Omega))}+$$$$ 
\gamma\Big)\Big\|\mathcal S^\nu\Big\|_{L^2_\#(\R,L^2(\partial\Omega))} +
\gamma\Big|\Omega\Big|\Big\|\nabla\mathcal S^\nu\Big\|_{L^2_\#(\R,L^2(\Omega))}
\end{equation}
Using (\ref{2.6}), Friedrichs's inequality and integrating (\ref{babou002}) with respect to $\theta \in[\theta_\alpha,\theta_\omega],$
\begin{equation}
\frac{\sqrt{\tilde G_{thr}}}{C+1}\Big\|\mathcal S^\nu\Big\|^2_{L^2_\#(\R,H^1(\Omega))}\leq
C\Big|\partial \Omega\Big|\Big((\gamma+\nu)\Big\|g\Big\|_{L^\infty_\#(\R,L^2(\Omega))}+$$$$ 
\gamma\Big)\Big\|\mathcal S^\nu\Big\|_{L^2_\#(\R,H^1(\Omega))} +
\gamma\Big|\Omega\Big|\Big\|\mathcal S^\nu\Big\|_{L^2_\#(\R,H^1(\Omega))}.
\end{equation}
Then 
\begin{equation}
\Big\|\mathcal S^\nu\Big\|_{L^2_\#(\R,H^1(\Omega))}\leq
\frac{C+1}{\sqrt{\tilde G_{thr}}}\Big[
C\Big|\partial \Omega\Big|\Big((\gamma+\nu)\Big\|g\Big\|_{L^\infty_\#(\R,L^2(\Omega))}+
\gamma\Big) +\gamma\Big|\Omega\Big|\Big].
\end{equation}
Or
\begin{equation}\label{gamma7}
\Big\|\mathcal S^\nu\Big\|_{L^2_\#(\R,H^1(\Omega))}\leq
\gamma_7.
\end{equation}
Integrating (\ref{babou00}) from 0 to $\theta$ and using the fact that the second term is positive, we have
\begin{equation}
\frac12\Big(\Big\| \mathcal S^\nu(\theta,\cdot)\Big\|^2_2\Big)- \frac12\Big(\Big\| \mathcal S^\nu(0,\cdot)\Big\|^2_2\Big)\leq 
\frac{1}{\epsilon^i} C\Big|\partial \Omega\Big|\Big((\gamma+\nu)\Big\|g\Big\|_{L^\infty_\#(\R,L^2(\Omega))}+$$$$ 
\gamma\Big)\Big\|\mathcal S^\nu\Big\|_{L^2_\#(\R,H^1(\Omega))} +\frac{1}{\epsilon^i}
\gamma\Big|\Omega\Big|\Big\|\mathcal S^\nu\Big\|_{L^2_\#(\R,H^1(\Omega))}.
\end{equation}
Taking the supremum for all $\theta\in (0,1)$ we get
\begin{equation}
\Big\| \mathcal S^\nu\Big\|^2_{L^\infty_\#(\mathbb R; L^2(\Omega))}\leq 
\frac{2\gamma_7}{\epsilon^i}\Big(C\Big|\partial \Omega\Big|\Big((\gamma+\nu)\Big\|g\Big\|_{L^\infty_\#(\R,L^2(\Omega))}+ 
\gamma\Big)+\gamma\Big|\Omega\Big|\Big)+\| z_0\|^2_2.\end{equation}
Multiplying (\ref{eqreg1}) by $-\Delta \mathcal S^\nu$ and integrating over $\Omega$ we get,
\begin{equation}\label{deltanu}
\frac12\frac{d}{d\theta}\int_{\Omega}\Big|\nabla\mathcal S^\nu\Big|^2dx-\int_{\partial\Omega}\frac{\partial\mathcal S^\nu}{\partial n}\frac{\partial \mathcal S^\nu}{\partial \theta}d\sigma+\frac{1}{\epsilon^i}\int_\Omega(\widetilde{\mathcal A}_\epsilon+\nu)\Big|\Delta\mathcal S^\nu\Big|^2dx=$$
$$-\frac{1}{\epsilon^i}\int_\Omega\nabla\cdot\widetilde{\mathcal C}_\epsilon\,\Delta\mathcal S^\nu dx-\frac{1}{\epsilon^i}\int_{\Omega}\nabla\widetilde{\mathcal A}_\epsilon\cdot\nabla\mathcal S^\nu\Delta \mathcal S^\nu dx,
\end{equation}
Using robin boundary condition, we get 
\begin{equation}
\frac12\frac{d}{d\theta}\int_{\Omega}\Big|\nabla\mathcal S^\nu\Big|^2dx+\frac12\frac{d}{d\theta}\int_{\partial\Omega}\Big|\mathcal S^\nu\Big|^2d\sigma
-\frac{d}{d\theta}\int_{\partial\Omega}\mathcal S^\nu d\sigma
+\frac{1}{\epsilon^i}\int_\Omega(\widetilde{\mathcal A}_\epsilon+\nu)\Big|\Delta\mathcal S^\nu\Big|^2dx$$
$$=-\frac{1}{\epsilon^i}\int_\Omega\nabla\cdot\widetilde{\mathcal C}_\epsilon\,\Delta\mathcal S^\nu dx-\frac{1}{\epsilon^i}\int_{\Omega}\nabla\widetilde{\mathcal A}_\epsilon\cdot\nabla\mathcal S^\nu\Delta \mathcal S^\nu dx,
\end{equation}
Integrating from 0 to 1 with respect to $\theta,$ and using (\ref{int1}) for $U=\Delta \mathcal S^\nu\,, V=\nabla\cdot\widetilde{\mathcal C}_\epsilon$ for the first term of the right hand side and $V=\nabla\widetilde{\mathcal A}_\epsilon\cdot\nabla\mathcal S^\nu\,,U=\Delta S^\nu$ for the second term of the right hand side, and using periodicity of $\theta\rightarrow\int_\Omega\Big|\nabla\mathcal S^\nu(\theta,\cdot)\Big|^2,$
we get:
\begin{equation}
\frac12\int_0^1\int_\Omega(\widetilde{\mathcal A}_\epsilon+\nu)\Big|\Delta\mathcal S^\nu\Big|^2dxd\theta\leq\int_0^1\int_{\Omega}\frac{|\nabla\widetilde{\mathcal A}_\epsilon|^2}{\widetilde{\mathcal A}_\epsilon+\nu}  |\nabla\mathcal S^\nu|^2dxd\theta+\int_0^1\int_\Omega\frac{|\nabla\cdot\widetilde{\mathcal C}_\epsilon|^2}{\widetilde{\mathcal A}_\epsilon+ \nu} dxd\theta,
\end{equation}
Taking $\theta$ in $[\theta_\alpha,\theta_\omega]$ we have 
\begin{equation}
\tilde G_{thr}\Big\|\Delta\mathcal S^\nu\Big\|^2_{L^2_\#(\R,L^2(\Omega))}
\leq \frac{\gamma^2}{\tilde G_{thr}}\Big\|\mathcal S^\nu\Big\|^2_{L^2_\#(\R,H^1(\Omega))}
+\frac{\gamma^2|\Omega|}{\tilde G_{thr}},
\end{equation}
giving 
\begin{equation}
\Big\|\Delta\mathcal S^\nu\Big\|^2_{L^2_\#(\R,L^2(\Omega))}
\leq \frac{\gamma^2}{\tilde G^2_{thr}}\Big(\gamma^2_7+1),
\end{equation}.\\
Multiplying (\ref{eqreg1}) by $\frac{\partial \mathcal S^\nu}{\partial\theta}$ and integrating with respect to the variable $x,$ we get
$$\int_\Omega\Big|\frac{\partial \mathcal S^\nu}{\partial \theta}\Big|^2dx
+\frac{1}{\epsilon^i}\int_\Omega (\widetilde{\mathcal A}_\epsilon+\nu)\nabla\mathcal S^\nu\cdot\nabla(\frac{\partial \mathcal S^\nu}{\partial \theta})dx
-\frac{1}{\epsilon^i}\int_{\partial\Omega} (\widetilde{\mathcal A}_\epsilon+\nu)\frac{\partial \mathcal S^\nu}{\partial n}\frac{\partial \mathcal S^\nu}{\partial \theta}d\sigma=$$
$$
\frac{1}{\epsilon^i}\int_\Omega\nabla\cdot\widetilde{\mathcal C}_\epsilon \frac{\partial \mathcal S^\nu}{\partial \theta}dx,$$
Using the fact that 
\begin{equation}            
\int_{\Omega}\big(\widetilde{\mathcal A}_\epsilon+\nu\big)\nabla\mathcal S^\nu\cdot\nabla \frac{\partial\mathcal S^\nu}{\partial\theta}dx=\frac12\frac{d\Big(\int_\Omega(\widetilde{\mathcal A}_\epsilon+\nu)|\nabla\mathcal S^\nu|^2dx\Big)}{d\theta}-\frac12\int_\Omega\frac{\partial\widetilde{\mathcal A}_\epsilon}{\partial\theta}|\nabla\mathcal S^\nu|^2dx,
\end{equation}
Since  $\widetilde{\mathcal A}^\epsilon,\,\,\frac{\partial\widetilde{\mathcal C}^\epsilon}{\partial\theta}$ and $\frac{\partial\widetilde{\mathcal A}_\epsilon}{\partial\theta}$ are bounded by $\gamma,$ we get
\begin{equation}\label{nurenouv}
\Big\|\frac{\partial \mathcal S^\nu}{\partial \theta}\Big\|^2_2+\frac{1}{\epsilon^i}\frac12\frac{d\Big(\int_\Omega
(\widetilde{\mathcal A}_\epsilon+\nu)|\nabla\mathcal S^\nu|^2dx\Big)}{d\theta}+
\frac{1}{\epsilon^i}\frac12\frac{d\Big(\int_{\partial\Omega}
(\widetilde{\mathcal A}_\epsilon+\nu)|\mathcal S^\nu|^2d\sigma\Big)}{d\theta}
\leq$$$$
\frac{1}{\epsilon^i}\frac\gamma2(\int_\Omega|\nabla\mathcal S^\nu|^2dx+\int_{\partial\Omega}|\mathcal S^\nu|^2d\sigma)+\frac{1}{\epsilon^i}
\int_{\partial\Omega}\Big(\big( \widetilde{\mathcal A}_\epsilon+\nu\big)g+\widetilde{\mathcal C}_\epsilon\cdot n\Big)\frac{\partial  \mathcal S^\nu}{\partial \theta}d\sigma+
$$$$\frac{1}{\epsilon^i}\gamma\int_{\Omega}\big|\nabla\mathcal S^\nu\big|dx+
\frac{1}{\epsilon^i}\gamma\frac{d}{d\theta}\int_{\Omega}\big|\nabla\mathcal S^\nu\big|dx
\end{equation}
or
\begin{equation}\label{nurenouv}
\Big\|\frac{\partial \mathcal S^\nu}{\partial \theta}\Big\|^2_2+\frac{1}{\epsilon^i}\frac12\frac{d\Big(\int_\Omega
(\widetilde{\mathcal A}_\epsilon+\nu)|\nabla\mathcal S^\nu|^2dx\Big)}{d\theta}+
\frac{1}{\epsilon^i}\frac12\frac{d\Big(\int_{\partial\Omega}
(\widetilde{\mathcal A}_\epsilon+\nu)|\mathcal S^\nu|^2d\sigma\Big)}{d\theta}
\leq$$$$
\frac{1}{\epsilon^i}\frac\gamma2(\int_\Omega|\nabla\mathcal S^\nu|^2dx+\int_{\partial\Omega}|\mathcal S^\nu|^2d\sigma)+\frac{1}{\epsilon^i}
\frac{d}{d\theta}\int_{\partial\Omega}\Big(\big( \widetilde{\mathcal A}_\epsilon+\nu\big)g+\widetilde{\mathcal C}_\epsilon\cdot n\Big) \mathcal S^\nu d\sigma+
$$$$\frac{1}{\epsilon^i}\gamma\int_{\Omega}\big|\nabla\mathcal S^\nu\big|dx+
\frac{1}{\epsilon^i}\gamma\frac{d}{d\theta}\int_{\Omega}\big|\nabla\mathcal S^\nu\big|dx+
\int_{\partial\Omega}\frac{\partial\widetilde{\mathcal A}_\epsilon}{\partial\theta}g+
\frac{\partial\widetilde{\mathcal C}_\epsilon}{\partial \theta}\cdot n\Big)\mathcal S^\nu d\sigma
\end{equation}
Integrating (\ref{nurenouv})\,
over $\theta\in [0, 1]$ and using 1-periodicity of $\mathcal S^\nu$ and $\nabla\mathcal S^\nu$ with respect to $\theta$ we have:
\begin{equation}
\Big\|\frac{\partial \mathcal S^\nu}{\partial \theta}\Big\|^2_{L^2_\#(\R,H^1(\Omega))}
\leq
\frac{1}{\epsilon^i}\frac\gamma2\int_0^1(\int_\Omega|\nabla\mathcal S^\nu|^2dx+\int_{\partial\Omega}|\mathcal S^\nu|^2d\sigma)d\theta
+
\frac{1}{\epsilon^i}\gamma\int_0^1\int_{\Omega}\big|\nabla\mathcal S^\nu\big|dxd\theta+$$$$
\int_0^1\int_{\partial\Omega}\Big(\frac{\partial\widetilde{\mathcal A}_\epsilon}{\partial\theta}g+
\frac{\partial\widetilde{\mathcal C}_\epsilon}{\partial \theta}\cdot n\Big)\mathcal S^\nu d\sigma d\theta
\end{equation}
hence
\begin{equation}
\Big\|\frac{\partial \mathcal S^\nu}{\partial \theta}\Big\|^2_{L^2_\#(\R,H^1(\Omega))}
\leq
\frac{1}{\epsilon^i}\frac\gamma2\gamma_7\Big((C+1)+2\Big|\Omega\Big|+2C\Big|\partial\Omega\Big|(\Big\|g\Big\|_{L^\infty_\#(\R,L^2(\Omega))}+1)
\Big)
\end{equation}
The solution $\mathcal S^\nu$ of the equation (\ref{eqreg1}) is differentiable over the time.   $\frac{\partial \mathcal S^\nu}{\partial t}$ is then solution to
\begin{equation}\label{bbbb}
 \frac{\partial\Big(\frac{\partial \mathcal S^\nu}{\partial t}\Big)}{\partial\theta}-\frac{1}{\epsilon^i}
 \nabla\cdot\big((\widetilde{\mathcal A}_\epsilon+\nu)\nabla(\frac{\partial \mathcal S^\nu}{\partial t})\big)
 =\frac{1}{\epsilon^i}\nabla\cdot\widetilde{\mathcal C}^\epsilon
\end{equation}
where
\begin{equation}\label{2.79}
\widetilde{\mathcal C}^\epsilon=
\frac{\partial \widetilde{\mathcal C}_\epsilon}{\partial t}+\frac{\partial \widetilde{\mathcal A}_\epsilon}{\partial t}\nabla\mathcal S^\nu,
\,\,i=0,1
\end{equation}
and
\begin{equation}\label{2.80}
\nabla\cdot\tilde{\mathcal C^\epsilon}=\frac{\partial(\nabla\cdot\widetilde{\mathcal C}_\epsilon)}{\partial t}+
\frac{\partial (\nabla\widetilde{\mathcal A}_\epsilon)}{\partial t}\nabla S^\nu+
\frac{\partial \widetilde{\mathcal A}_\epsilon}{\partial t}\Delta S^\nu
\end{equation}
Since the coefficients $\widetilde{\mathcal A}_\epsilon$ and $\widetilde{\mathcal C}_\epsilon$, and above estimates, 
 are bounded by $\gamma$, then coefficients  $\widetilde{\mathcal C}^\epsilon$ defined by (\ref{2.79}) and 
 $\nabla\cdot\widetilde{\mathcal C}^\epsilon$  defined by (\ref{2.80}) are also bounded by a constant 
 $\epsilon^i\gamma_8,\, i=0,1,2$ depending only on $\gamma_7$.
 \begin{equation}\left\{\begin{array}{ccc}
\Big\| \tilde{\mathcal C^\epsilon} \Big\|_{L^2_\#(\mathbb R,L^2(\Omega))}\leq \epsilon^i\gamma_8  \\
\Big\| \nabla\cdot\tilde{\mathcal C^\epsilon} \Big\|_{L^2_\#(\mathbb R,L^2(\Omega))}\leq \epsilon^i\gamma_8.
\end{array}\right.\end{equation}
Note that the constant $\gamma_8$ is a generic constant.  Its value may vary depending on the short-mean terms and the long term one. In accordance with the work of Faye et al \cite{FaFreSe} and to the hypotheses given on (\ref{epsmu1}),  (\ref{epsmu2}) and (\ref{0beq3}) on current velocity fied and tidal oscillation, in the  case  of short and mean term, $\gamma_8$ is a constant not depending on $\epsilon$ and in the case of long term, it depends on $\epsilon$ and tends to zero when the $\epsilon\rightarrow0.$\\
Multiplying (\ref{bbbb}) by $\frac{\partial S^\nu}{\partial t}$ and integrating over $\Omega,$ 
we get
\begin{equation}\label{revise}
\frac12\frac{d}{d\theta}(\int_\Omega \big|\frac{\partial \mathcal S^\nu}{\partial t}\Big|^2dx)
+\frac{1}{\epsilon^i}\int_{\partial\Omega} (\widetilde{\mathcal A}_\epsilon+\nu)\Big|\frac{\partial \mathcal S^\nu}{\partial t}\Big|^2d\sigma
+\frac{1}{\epsilon^i}\int_\Omega (\widetilde{\mathcal A}_\epsilon+\nu)\Big|\nabla(\frac{\partial \mathcal S^\nu}{\partial t})\Big|^2dx-$$
$$
\frac{1}{\epsilon^i}\int_{\partial\Omega} (\widetilde{\mathcal A}_\epsilon+\nu)\frac{\partial(\frac{\partial \mathcal S^\nu}{\partial t})}{\partial n}\frac{\partial \mathcal S^\nu}{\partial t}d\sigma=
\frac{1}{\epsilon^i}\int_\Omega \nabla\cdot\tilde{\mathcal C^\epsilon}\frac{\partial \mathcal S^\nu}{\partial t}dx.\end{equation}
From this last equality we have
\begin{equation}\label{br}
\frac12\frac{d}{d\theta}( \Big\|\frac{\partial \mathcal S^\nu}{\partial t}\Big\|^2_2)
++\frac{1}{\epsilon^i}\int_{\partial\Omega} (\widetilde{\mathcal A}_\epsilon+\nu)\Big|\frac{\partial \mathcal S^\nu}{\partial t}\Big|^2d\sigma
+\frac{1}{\epsilon^i}\int_\Omega (\widetilde{\mathcal A}_\epsilon+\nu)\Big|\nabla(\frac{\partial \mathcal S^\nu}{\partial t})\Big|^2dx=$$
$$
\frac{1}{\epsilon^i}\int_{\partial\Omega} (\widetilde{\mathcal A}_\epsilon+\nu)\frac{\partial g}{\partial t}+
\widetilde{\mathcal C}^\epsilon\cdot n\Big)\frac{\partial \mathcal S^\nu}{\partial t}d\sigma
-\frac{1}{\epsilon^i}\int_\Omega \tilde{\mathcal C^\epsilon}\cdot\nabla\frac{\partial \mathcal S^\nu}{\partial t}dx
\end{equation}
integrating 
 (\ref{br}) from 0 to 1 with respect to $\theta$ and using 1-periodicity of $\theta\mapsto \Big\|\frac{\partial \mathcal S^\nu}{\partial t}(\theta,\cdot)\Big\|^2_2,$
 we have
\begin{equation}
\int_0^1\int_{\partial\Omega}\widetilde{\mathcal A}_\epsilon\Big|\frac{\partial \mathcal S^\nu}{\partial t}\Big|^2d\sigma
+\int_\Omega \widetilde{\mathcal A}_\epsilon\Big|\nabla(\frac{\partial \mathcal S^\nu}{\partial t})\Big|^2dxd\theta=$$
$$
C\Big|\partial\Omega\Big|\Big( (\gamma+\nu)\Big\|\frac{\partial g}{\partial t}\Big\|_{L^\infty_\#(\mathbb R, L^2(\Omega))}+
\epsilon^i\gamma_8\Big)\Big\|\mathcal S^\nu\Big\|_{L^2_\#(\mathbb R, H^1(\Omega))}
+\epsilon^i\gamma_8\Big\|\mathcal S^\nu\Big\|_{L^2_\#(\mathbb R, H^1(\Omega))}
\end{equation} 
which gives for $\theta$ in $[\theta_\alpha, \theta_\omega]$
\begin{equation}\label{2.83}
\frac{\tilde G_{thr}}{C+1}\Big\|\frac{\partial \mathcal S^\nu}{\partial t}\Big\|^2_{L^2_\#(\mathbb R, H^1(\Omega))}\leq
C\Big|\partial\Omega\Big|\Big( (\gamma+\nu)\Big\|\frac{\partial g}{\partial t}\Big\|_{L^\infty_\#(\mathbb R, L^2(\Omega))}+
\epsilon^i\gamma_8\Big)\Big\|\mathcal S^\nu\Big\|_{L^2_\#(\mathbb R, H^1(\Omega))}$$$$
+\epsilon^i\gamma_8\Big\|\mathcal S^\nu\Big\|_{L^2_\#(\mathbb R, H^1(\Omega))},
\end{equation}
Hence
\begin{equation}\label{utile}
\Big\|\frac{\partial \mathcal S^\nu}{\partial t}\Big\|_{L^2_\#(\mathbb R, H^1(\Omega))}\leq
\frac{\gamma_7(C+1)}{\tilde G_{thr}}\Big(
C\Big|\partial\Omega\Big|\Big( (\gamma+\nu)\Big\|\frac{\partial g}{\partial t}\Big\|_{L^\infty_\#(\mathbb R, L^2(\Omega))}+
\epsilon^i\gamma_8\Big)
+\epsilon^i\gamma_8\Big),
\end{equation}
where $\tilde\gamma_8$ depends only on $\gamma8.$\\
\end{proof}
Now we get estimates of $\mathcal S^\nu,$ which do not depends on $\frac1\epsilon.$
\begin{theorem} 
Under the same assumptions as in theorem \ref{1th2}, the solution $\mathcal S^\nu=\mathcal S^\nu(t,\theta,x)$ of (\ref{eqreg1}) satisfies the folllowing properties\begin{equation}\label{nouvth2}
\Big\|\mathcal S^\nu\Big\|^2_{L^2_\#(\R,H^1(\Omega))}\leq \gamma^2_7+
2(C+1)\frac{\gamma^2}{G_{thr}}\Big(\gamma^2_7+1\Big).
\end{equation}
$\tilde\gamma_4$ depends on $\Omega, \,\,z_0\,$ and $g.$
\end{theorem}
\begin{proof}
Multiplying (\ref{eqreg1}) by $-(\Delta \mathcal S^\nu)$ and integrating over $\Omega$ we get (\ref{deltanu}). 
%
Using (\ref{int1}) for $U=\Delta \mathcal S^\nu\,, V=\nabla\cdot\widetilde{\mathcal C}_\epsilon$ for the first term of the right hand side and $V=\nabla\widetilde{\mathcal A}_\epsilon\cdot\nabla\mathcal S^\nu,\,U=\Delta S^\nu$ for the second term of the right hand side, 
we get:
\begin{equation}
\frac12\frac{d}{d\theta}\int_\Omega \Big|\nabla \mathcal S^\nu\Big|^2dx+
\frac12\frac{d}{d\theta}\int_{\partial\Omega} \Big|\mathcal S^\nu\Big|^2d\sigma
-\int_{\partial\Omega}g\frac{\partial \mathcal S^\nu}{\partial \theta}d\sigma+\frac{1}{2\epsilon^i}\int_\Omega(\widetilde{\mathcal A}_\epsilon+\nu)\Big|\Delta\mathcal S^\nu\Big|^2dx$$
$$\leq\frac{1}{\epsilon^i}\int_{\Omega}\frac{|\nabla\widetilde{\mathcal A}_\epsilon|^2}{\widetilde{\mathcal A}_\epsilon+\nu}  |\nabla\mathcal S^\nu|^2dx+\frac{1}{\epsilon^i}\int_\Omega\frac{|\nabla\cdot\widetilde{\mathcal C}_\epsilon|^2}{\widetilde{\mathcal A}_\epsilon+ \nu} dx,
\end{equation}
Integrating from $\theta_0$ giving in (\ref{2.70}) to $\theta\in[0,1]$ and using Friedrichs's inequality, we have 
\begin{equation}\label{nnouv1}
\frac12 \Big\|\mathcal S^\nu(\theta, \cdot)\Big\|^2_{H^1(\Omega)}\leq \frac12\Big\|\mathcal S^\nu(\theta_0, \cdot)\Big\|^2_{H^1(\Omega)}+
(C+1)\Big[\frac{1}{\epsilon^i}\int_0^1\int_{\Omega}\frac{|\nabla\widetilde{\mathcal A}_\epsilon|^2}{\widetilde{\mathcal A}_\epsilon+\nu}  |\nabla\mathcal S^\nu|^2dx\,d\theta+$$
$$\frac{1}{\epsilon^i}\int_0^1\int_\Omega\frac{|\nabla\cdot\widetilde{\mathcal C}_\epsilon|^2}{\widetilde{\mathcal A}_\epsilon+ \nu} dxd\theta\Big]
\end{equation}
From (\ref{hupp1}) we have,  $|\nabla\widetilde{\mathcal A}_\epsilon\leq\sqrt{\epsilon^i}\gamma,\,\,|\nabla\widetilde{\mathcal C}_\epsilon\leq\sqrt{\epsilon^i}\gamma$ 
and (\ref{gamma7}), then
Inequality (\ref{nnouv1}) becomes 
\begin{equation}\label{nnouv3}
\Big\|\mathcal S^\nu(\theta, \cdot)\Big\|^2_{H^1(\Omega)}\leq \gamma^2_7+
2(C+1)\frac{\gamma^2}{G_{thr}}\Big(\gamma^2_7+1\Big).
\end{equation}
Taking the suppremum for $\theta\in[0,1],$ we get 
\begin{equation}\label{utile0}
\Big\|\mathcal S^\nu\Big\|^2_{L^\infty_\#(\R,H^1(\Omega))}\leq \gamma^2_7+
2(C+1)\frac{\gamma^2}{G_{thr}}\Big(\gamma^2_7+1\Big).
\end{equation}
giving also (\ref{nouvth2}).
\end{proof}
\noindent Passing to the limit as $\nu$ tends to zero we obtain the following theorem.
\begin{theorem}\label{biba2} Under the same assumptions as in theorem\,\ref{1th2},
$\forall \epsilon>0\,\,$ there exists a unique $\mathcal S=\mathcal S(t,\theta,x)$ periodic of period 1 solution to
\begin{equation}\label{biba}
\left\{\begin{array}{ccc}
\frac{\partial\mathcal S}{\partial \theta}-\frac{1}{\epsilon^i}\nabla\cdot\Big(\widetilde{\mathcal A}_\epsilon\nabla 
\mathcal S\Big)=\frac{1}{\epsilon^i}\nabla \cdot\widetilde{\mathcal C}_\epsilon\,\,\text{in}\,\,(0,T)\times\mathbb
R\times\Omega\,\,\,i=0,1\\
 \mathcal S(0,0,x)=z_0(x)\,\,\text{in}\,\,\Omega\\
 \frac{\partial \mathcal S}{\partial n}+\mathcal S= g\,\,\text{on}\,\,\quad(0,T)\times\mathbb R\times\partial\Omega.
\end{array}\right.\end{equation}
Moreover, the following estimates hold.
\begin{equation}\label{S1}
\sup_{\theta\in\mathbb R}\Big|\int_\Omega\mathcal S(\theta,x)dx\Big|=0
\end{equation}
\begin{equation}\label{S2}
\Big\|\mathcal S\Big\|^2_{L^\infty_\#(\R,H^1(\Omega))}\leq \gamma^2_7+
2(C+1)\frac{\gamma^2}{G_{thr}}\Big(\gamma^2_7+1\Big).
\end{equation}
\begin{equation}\label{S3}
\Big\|\frac{\partial \mathcal S}{\partial t}\Big\|_{L^2_\#(\mathbb R, H^1(\Omega))}\leq\gamma_9
\end{equation}
which $\gamma_9$ depending on $\gamma_8$ and $\epsilon^i$
\end{theorem}
\begin{proof}
To show that solution (\ref{biba}) is unique, let  $\mathcal S$ and
$\bar{\mathcal S}$ two solutions of (\ref{biba}). Then, $\mathcal S-\bar{\mathcal S}$ is solution of
\begin{equation}
\left\{\begin{array}{ccc}
\frac{\partial(\mathcal S-\bar{\mathcal S})}{\partial \theta}-\frac{1}{\epsilon^i}\nabla\cdot\Big(\widetilde{\mathcal A}_\epsilon\nabla (\mathcal S-\bar{\mathcal S})\Big)=0\,\,\text{in}\,\,(0,T)\times\mathbb R\times\Omega\\
 \mathcal S(0,0,x)-\bar{\mathcal S}(0,0,x)=0\,\,\text{in}\,\,\Omega\\
 \frac{\partial (\mathcal S-\bar{\mathcal S})}{\partial n}= -(\mathcal S-\bar{\mathcal S})\,\,\text{on}\,\,(0,T)\times\mathbb R\times\partial\Omega.
\end{array}\right.\end{equation}
Multiplying (\ref{biba}) by $\mathcal S-\bar{\mathcal S}$ and integrating over $\Omega$,
we get:
$$\frac12\frac{d}{d\theta}\int_{\Omega}\Big|(\mathcal S-\bar{\mathcal S})\Big|^2dx+\frac{1}{\epsilon^i}
\int_\Omega \widetilde{\mathcal A}_\epsilon\Big|\nabla(\mathcal S-\bar{\mathcal S})\Big|^2dx-\frac{1}{\epsilon^i}
\int_{\partial\Omega}\widetilde{\mathcal A}_\epsilon\frac{\partial (\mathcal S-\bar{\mathcal S})}{\partial n}{(\mathcal S-\bar{\mathcal S})} d\sigma
=0.$$
Taking account the boundary condition , we have
$$\frac12\frac{d}{d\theta}\int_{\Omega}\Big|(\mathcal S-\bar{\mathcal S})\Big|^2dx+\frac{1}{\epsilon^i}
\int_\Omega \widetilde{\mathcal A}_\epsilon\Big|\nabla(\mathcal S-\bar{\mathcal S})\Big|^2
+\frac{1}{\epsilon^i}
\int_{\partial\Omega} \widetilde{\mathcal A}_\epsilon\Big|(\mathcal S-\bar{\mathcal S})\Big|^2
=0.$$
Since the solutions $\mathcal S$ et $\bar{\mathcal S}$ are periodic of period 1 and using Friedrichs's Inequality, we get
$$\frac{G_{thr}}{C+1}\Big\|(\mathcal S(\theta)-\bar{\mathcal S}(\theta))\Big\|^2_{L^2_\#(\R,H^1(\Omega))}\leq0
$$
Finally
$$
\mathcal S=\bar{\mathcal S}
$$
The inequalities (\ref{S1}) , (\ref{S2}) and (\ref{S3}) are obtained by letting $\nu$  tend towards 0 in the inequalities (\ref{bi}), (\ref{utile0})  
and (\ref{utile}).
\end{proof}
Having the estimates given in Theorem \ref{biba2}, we are able to give the proof of Theorem \ref{th1.0}.\\
\begin{proof}\,\,\textbf{of theorem\,\ref{th1.0}}
Let $z^\epsilon_1$ and
$z^\epsilon_2$ be two solutions of (\ref{shortlong1}), then, $z^\epsilon_1-z^\epsilon_2$ is solution to
\begin{equation}\label{biba1}\left\{\begin{array}{ccc}\frac{\partial (z^\epsilon_1-z^\epsilon_2)}{\partial t}-
\frac{1}{\epsilon^i}\nabla\cdot\Big(\mathcal A^\epsilon\nabla (z^\epsilon_1-z^\epsilon_2)\Big)=0 \,\,\text{in}\,\,(0,T)\times\Omega,\,\,\,i=1,2\\
z^\epsilon_1(0,x)-z^\epsilon_2(0,x)=0\,\,\text{in}\,\,\Omega\\
\frac{\partial (z^\epsilon_1-z^\epsilon_2)}{\partial n}= -(z^\epsilon_1-z^\epsilon_2)\,\,\text{on}\,\,(0,T)\times\partial \Omega.\end{array}\right.\end{equation}
Multiplying (\ref{biba1}) by $z^\epsilon_1-z^\epsilon_2$ and integrating over $\Omega$,
we get:
$$\frac12\frac{d}{dt}\int_{\Omega}\Big|(z^\epsilon_1-z^\epsilon_2)\Big|^2dx+
\frac{1}{\epsilon^i}\int_\Omega \mathcal A^\epsilon\Big|\nabla(z^\epsilon_1-z^\epsilon_2)\Big|^2dx-
\frac{1}{\epsilon^i}\int_{\partial\Omega}\mathcal A^\epsilon\frac{\partial (z^\epsilon_1-z^\epsilon_2)}{\partial n}{(z^\epsilon_1-z^\epsilon_2)} d\sigma
=0.$$
Taking into account the robin boundary condition, we obtain
$$\frac12\frac{d}{dt}\int_{\Omega}\Big|(z^\epsilon_1-z^\epsilon_2)\Big|^2dx+
\frac{1}{\epsilon^i}\int_\Omega \mathcal A^\epsilon\Big|\nabla(z^\epsilon_1-z^\epsilon_2)\Big|^2+
\frac{1}{\epsilon^i}\int_\Omega \mathcal A^\epsilon\Big|(z^\epsilon_1-z^\epsilon_2)\Big|^2
=0.$$
which gives as the second and third terms are positive,
\begin{equation}\label{babou005}
\frac{d(\Big\|z^\epsilon_1(t,\cdot)-z^\epsilon_2(t,\cdot)\Big\|^2)}{dt}\leq0
\end{equation}
Integrating (\ref{babou005}) over $t\in [0; T)$, we get
\begin{equation}
\Big\|z^\epsilon_1(t,\cdot)-z^\epsilon_2(t,\cdot)\Big\|^2\leq0
\end{equation}
then
\begin{equation}
\Big\|z^\epsilon_1(t,\cdot)-z^\epsilon_2(t,\cdot)\Big\|^2=0
\end{equation}
Finally $$z^\epsilon_1=z^\epsilon_2$$
Now, we consider the function $Z^\epsilon=Z^\epsilon(t,x)=\mathcal S(t,\frac{t}{\epsilon},x)$ where
$S$ is solution to (\ref{biba}) 
Since
\begin{equation}
\frac{\partial Z^\epsilon}{\partial t}=\frac{\partial\mathcal S}{\partial t}(t,\frac{t}{\epsilon},x)
+\frac{1}{\epsilon}\frac{\partial \mathcal S}{\partial \theta}(t,\frac{t}{\epsilon},x),
\end{equation}
We deduce from (\ref{biba}) that $Z^\epsilon$ is solution to
\begin{equation}\label{biba3}
\frac{\partial Z^\epsilon}{\partial t}-\frac{1}{\epsilon^{1+i}}\nabla\cdot(\mathcal A^\epsilon\nabla Z^\epsilon) =
\frac{1}{\epsilon^{1+i}}\nabla\cdot\mathcal C^\epsilon+\frac{\partial \mathcal S}{\partial t}(t,\frac{t}{\epsilon},x),\,\,i=0\,\,\text{or}\,\,1.
\end{equation}
We have to notice that, the case $i=0$ corresponds to short and mean term model. If $i=1,$ it corresponds to the long term model.\\
From relations  (\ref{shortlong1}) and (\ref{biba3}), we deduce that $z^\epsilon-Z^\epsilon$ is solution of
\begin{equation}\label{biba4}
\left\{
\begin{array}{cclll}
\frac{\partial (z^\epsilon-Z^\epsilon)}{\partial t}-\frac{1}{\epsilon^i}\nabla\cdot(\mathcal A^\epsilon \nabla (z^\epsilon-Z^\epsilon) & = & \frac{\partial \mathcal S}{\partial t}(t,\frac{t}{\epsilon},x) \,\,\text{in}\,\,(0,T)\times\Omega\\
 (z^\epsilon-Z^\epsilon)_{|t=0}& = & z_0-\mathcal S(0,0,\cdot)  \,\,\text{in}\,\,\Omega\\
\frac{\partial (z^\epsilon-Z^\epsilon)}{\partial n}+(z^\epsilon-Z^\epsilon)=0 \,\,\text{on}\,\,(0,T)\times\partial \Omega\end{array} \right.\end{equation}
Multiplying (\ref{biba4}) by $z^\epsilon-Z^\epsilon$ and integrating over $\Omega$, we get with the previous estimates
the following inequality:
\begin{equation}
\frac{d(\Big\|z^\epsilon-Z^\epsilon\Big\|_2^2)}{dt}\leq\gamma_9\Big\| z^\epsilon-Z^\epsilon\Big\|_2,
\end{equation}
Then
\begin{equation}\label{eqref}
 \Big\| z^\epsilon(t,\cdot)-Z^\epsilon(t,\cdot)\Big\|_2\leq\Big\| z_0-S(0,0,\cdot)\Big\|_2+\frac12\gamma_9t.
\end{equation}
We have  \begin{equation}\Big\| z^\epsilon(t,\cdot)\Big\|_2=\Big\| z^\epsilon(t,\cdot)-Z^\epsilon(t,\cdot)+Z^\epsilon(t,\cdot)\Big\|_2\leq\Big\| z^\epsilon(t,\cdot)-Z^\epsilon(t,\cdot)\Big\|_2+    \Big\|Z^\epsilon(t,\cdot)\Big\|_2\end{equation}
Taking the supremum for $t\in [0,T[$ 
and using (\ref{eqref}) we get
\begin{equation}
\Big\| z^\epsilon\Big\|_{L^\infty([0,T), L^2(\Omega))}\leq \Big\| z_0-S(0,0,\cdot)\Big\|_2+
\frac12\gamma_9 T+\sqrt{\gamma^2_7+2(C+1)\frac{\gamma^2}{G_{thr}}\Big(\gamma^2_7+1\Big)}.
\end{equation}
Finally
$$
\tilde \gamma=\Big\| z_0-S(0,0,\cdot)\Big\|_2+
\frac12\gamma_9 T+\sqrt{\gamma^2_7+2(C+1)\frac{\gamma^2}{G_{thr}}\Big(\gamma^2_7+1\Big)}.
$$
\end{proof}
\begin{remark}  \label{rq1}We can also replace the boundary condition  in (\ref{shortlong1}) by a Dirichlet condition
$z^\epsilon(t,x)=\tilde g(t,x)\,\,\partial\Omega.$  
where $\tilde g\in L^2([0, T), H^1(\Omega)).$\\
In this case, the solution $z^\epsilon$ to (\ref{shortlong1})  belong bounded in $L^\infty([0,T), H^1(\torus^2))$.  We no longer have an integral on the bounder to be managed. So the method developed in Faye et al \cite{FaFreSe}, can be applied without issues. 
\section{Homogenization and corrector results}
\subsection{On Two-Scale convergence}\label{OTSC}
In this subsection we are going to recall the notion of Two-Scale convergence.
\begin{definition}
A sequence of functions  $(z^\epsilon)$ in $L^\infty([0,T),L^2(\Omega))$ is said to Two-Scale converge to
$U\in L^{\infty}([0,T),L^{\infty}_\#(\R,L^2(\Omega)))$
 if for every $\psi\in\mathcal C([0,T),\mathcal C_{\#}(\R,\mathcal C(\Omega)))$ we have
\begin{gather}
\lim_{\epsilon\rightarrow 0}\int_{\Omega}\int_0^Tz^\epsilon(t,x)\psi(t,\frac t\epsilon,x)dt\,dx=\int_{\Omega}\int_0^T\int_0^1U(t,\theta,x)\psi(t,\theta,x)d\theta\, dt\,dx.
\end{gather}
\end{definition}
\noindent In \cite{allaire:1992} and \cite{nguetseng:1989}, the following theorem is given.
\begin{theorem} If a sequence $(z^\epsilon)$ is bounded in $L^\infty([0,T),L^2(\torus^2)),$ there exists  a subsequence still denoted $(z^\epsilon)$ and  a function $U\in L^{\infty}([0,T),L^{\infty}_\#(\R,L^2(\Omega))) $
such that
\begin{gather}
z^\epsilon\longrightarrow U\,\,\textrm{Two-Scale}.
\end{gather}
\end{theorem}
~\\
\subsection{Homogenization results}
In this section, we present the mathematical two-scale convergence and homogenization in the short, mean and long terms models. In other words, we consider equation (\ref{short1}) or (\ref{long1}) where  $\mathcal A^\epsilon$ and $\mathcal C^\epsilon$ are defined by (\ref{coef A}) coupled with (\ref{epsmu2})  or (\ref{aceps})  coupled (\ref{0beq3}). Then,
our aim is to determine the limiting behavior when $\epsilon\rightarrow 0,$ of the sequence$(z^\epsilon)_{\epsilon>0}$ of solutions to either (\ref{short1}) or (\ref{long1}). \\
In the short and mean term, the sequence 
\begin{equation}\label{coef1}
\mathcal A^\epsilon(t,x) \,\,\text{Two-Scale converges to}\,\,\widetilde{\mathcal A}(t, \theta,\tau,x)\in L^\infty([0,T], L^\infty_\#(\mathbb R, L^2(\Omega)))\,\,\text{and} $$
$$\mathcal C^\epsilon(t,x) \,\,\text{Two-Scale converges to}\,\,\widetilde{\mathcal C}(t, \theta,\tau,x)\in (L^\infty([0,T], L^\infty_\#(\mathbb R, L^2(\Omega))))^2
\end{equation}
with \begin{equation}\label{coef2}
\widetilde{\mathcal A}(t, \theta,\tau,x)= ag_a(|\mathcal U(t, \theta,\tau,x)|)\,\,\text{and}$$
$$\,\,\widetilde{\mathcal C}(t, \theta,\tau,x)=cg_c(\mathcal U(t, \theta,\tau,x))\frac{\mathcal U(t, \theta,\tau,x)}{|\mathcal U(t, \theta,\tau,x)|}
\end{equation}
where $\mathcal U$ is given in (\ref{epsmu1}) or (\ref{epsmu2}).
In the long term model, we have
\begin{equation}\label{coef3}
\mathcal A^\epsilon(t,x) \,\,\text{Two-Scale converges to}\,\,\widetilde{\mathcal A}(t, \theta,x)\in L^\infty([0,T], L^\infty_\#(\mathbb R, L^2(\Omega)))\,\,\text{and}$$
$$ \mathcal C^\epsilon(t,x) \,\,\text{Two-Scale converges to}\,\,\widetilde{\mathcal C}(t, \theta,\tau,x)\in (L^\infty([0,T], L^\infty_\#(\mathbb R, L^2(\Omega))))^2,
\end{equation}
with 
\begin{equation}\label{coef4} \widetilde{\mathcal A}(t,\theta,x)= ag_a(|\mathcal U_0(\theta)|)\,\,\text{and}\,\,\widetilde{\mathcal C}(t,\theta,x)=cg_c(|\mathcal U_0(\theta)|)\frac{\mathcal U_0(\theta)}{|\mathcal U_0(\theta)|}
\end{equation}
where 
$\mathcal U_0$ is given in (\ref{0beq3}).
We have the following theorem concerning the short and mean term models.
\begin{theorem}\label{th2}
Under assumptions  (\ref{hyp1}), (\ref{hupp1}), (\ref{2.6}), (\ref{2.7}), (\ref{coef1}) and (\ref{coef2}) and  for any $T,$ not depending on $\epsilon,$  the sequence $(z^{\epsilon})$ of solutions to 
(\ref{short1}), with coefficients given by (\ref{homoeq})-(\ref{homoeq3}), Two-Scale converges to the profile
$U\in L^{\infty}([0,T],L^{\infty}_\#(\R,L^2(\Omega)))$ solution to
\begin{equation}
\label{eqlim2} \left\{\begin{array}{ccc}
\frac{\partial U}{\partial\theta}
-\nabla\cdot(\widetilde{\mathcal{A}}\nabla U)=\nabla \cdot\widetilde{\mathcal{C}}\,\,\text{in}\,\,(0,T)\times\mathbb R\times \Omega\\
\frac{\partial U}{\partial n}+U= g\,\,\text{on}\,\,(0,T)\times\mathbb R\times \partial\Omega\end{array}\right.
\end{equation} where $\widetilde{\mathcal{A}}$ and $\widetilde{\mathcal{C}}$ are given by (\ref{coef2}).
\end{theorem}
\begin{proof}
Let $\psi^\epsilon(t,x)=\psi(t,\frac{t}{\sqrt{\epsilon}},\frac t\epsilon,x)$  be a regular function with compact support on
$[0,T)\times\Omega$ and periodic of period 1. Multiplying (\ref{short1}) by $\psi^\epsilon$ and integrating  over $[0,T)\times\Omega$ we get :
\begin{equation}\label{vi}
\int_\Omega \int^T_0\frac{\partial z^\epsilon}{\partial t}\psi^\epsilon dtdx-\frac{1}{\epsilon}\int_\Omega\int^T_0\nabla\cdot(\mathcal A^{\epsilon}\nabla z^{\epsilon})\psi^\epsilon dtdx =\frac{1}{\epsilon} \int_{\Omega}\int^T_0\nabla\cdot\mathcal C^\epsilon\psi^\epsilon dtdx.
\end{equation}
 Using integration by parts over $[0,T)$ in the first term and Green  formula over $\Omega$ in the second integral, we get
\begin{equation}
 -\int_{\Omega} z_0(x)\psi(0,0,0,x) dx-\int_{\Omega}\int^T_0\frac{\partial\psi^\epsilon }{\partial t} z^\epsilon dtdx+\frac{1}{\epsilon}\int_{\Omega}\int^T_0 \mathcal A^{\epsilon}\nabla z^{\epsilon}\nabla\psi^\epsilon dtdx$$
 $$-\frac{1}{\epsilon}\int_0^T\int_{\partial\Omega} \mathcal A^\epsilon\frac{\partial z^\epsilon}{\partial n}\psi^\epsilon d\sigma=-\frac{1}{\epsilon} \int_{\Omega}\int^T_0\mathcal C^\epsilon\cdot\nabla\psi^\epsilon dtdx+\frac{1}{\epsilon}\int_0^T\int_{\partial\Omega} \mathcal C^\epsilon\psi^\epsilon.nd\sigma.
\end{equation}
But $\frac{\partial \psi^\epsilon}{\partial t
}$ writes
\begin{equation}\label{deriv02}
\frac{\partial \psi^\epsilon}{\partial t
}=\Big(\frac{\partial\psi}{\partial t}\Big)^\epsilon+\frac{1}{\sqrt{\epsilon}}\Big(\frac{\partial\psi}{\partial \tau}\Big)^\epsilon+
\frac{1}{\epsilon}\Big(\frac{\partial\psi}{\partial \theta}\Big)^\epsilon,
\end{equation}
where
\begin{equation}
\Big(\frac{\partial\psi}{\partial t}\Big)^\epsilon(t,x)=\frac{\partial\psi}{\partial t}(t,\frac{1}{\sqrt{\epsilon}},\frac{t}{\epsilon},x),\,\,
\Big(\frac{\partial\psi}{\partial \tau}\Big)^\epsilon(t,x)=\frac{\partial\psi}{\partial \tau}(t,\frac{1}{\sqrt{\epsilon}},\frac{t}{\epsilon},x)$$
$$
\,\,\text{and}\,\,\Big(\frac{\partial\psi}{\partial \theta}\Big)^\epsilon(t,x)=\frac{\partial\psi}{\partial \theta}(t,\frac{1}{\sqrt{\epsilon}},\frac t\epsilon,x)
\end{equation}
Thus, because of  (\ref{deriv02}) and hypothesis (\ref{2.7}) we get
\begin{equation}
 \int_{\Omega}\int^T_0 z^\epsilon\Big(\Big(\frac{\partial\psi}{\partial t}\Big)^\epsilon+\frac{1}{\sqrt{\epsilon}}\Big(\frac{\partial\psi}{\partial \tau}\Big)^\epsilon+
\frac{1}{\epsilon}\Big(\frac{\partial\psi}{\partial \theta}\Big)^\epsilon+\frac{1}{\epsilon}\mathcal A^{\epsilon}\nabla z^\epsilon\nabla \psi^\epsilon)\Big) dtdx$$
$$ =-\frac{1}{\epsilon} \int_{\Omega}\int^T_0\mathcal C^\epsilon\cdot\nabla\psi^\epsilon dtdx-\int_{\Omega} z_0(x)\psi(0,0,x) dx.
\end{equation}
Multiplying by $\epsilon,$  we get finally
\begin{equation}
\epsilon \int_{\Omega}\int^T_0 z^\epsilon\Big(\frac{\partial\psi}{\partial t}\Big)^\epsilon dtdx+
\sqrt{\epsilon}\int_{\Omega}\int^T_0\Big(\frac{\partial\psi}{\partial \tau}\Big)^\epsilon z^\epsilon dtdx+
\int_{\Omega}\int^T_0\Big(\frac{\partial\psi}{\partial \theta}\Big)^\epsilon z^\epsilon dtdx$$
$$+ \int_{\Omega}\int^T_0\nabla\cdot(\mathcal A^{\epsilon}\nabla \psi^\epsilon) z^\epsilon dtdx =- \int_{\Omega}\int^T_0\mathcal C^\epsilon\cdot\nabla\psi^\epsilon dtdx-\epsilon\int_{\Omega} z_0(x)\psi(0,0,x) dx.
\end{equation}
As $\psi^\epsilon$ is regular with compact support on $[0,T)\times\Omega,$ and $\mathcal A^\epsilon$ is a regular function, the functions
$\Big(\frac{\partial\psi}{\partial t}\Big)^\epsilon,\,\,\Big(\frac{\partial\psi}{\partial \tau}\Big)^\epsilon,\,\,
\Big(\frac{\partial\psi}{\partial \theta}\Big)^\epsilon,
\,\,\nabla\cdot(\mathcal A^{\epsilon}\nabla \psi^\epsilon)\Big)$ and $\nabla\psi^\epsilon$ can be considered as test functions. Then  using  two-scale convergence we get  when $\epsilon\rightarrow0,$
\begin{equation}\label{vie0}
\int_0^1\int_{\Omega}\int^T_0\frac{\partial\psi}{\partial \theta} U dtd\theta dx+ \int_0^1\int_{\Omega}\int^T_0\nabla\cdot(\widetilde{\mathcal A}\nabla \psi)\Big) U dtd\theta dx $$
$$=- \int_0^1\int_{\Omega}\int^T_0\widetilde{\mathcal C}\cdot\nabla\psi dtd\theta dx.
\end{equation}
Using Green Formula, we get
\begin{equation}\label{vie1}
\int_{\Omega}\int_0^1\int^T_0\Big(\frac{\partial U}{\partial \theta} -\nabla\cdot(\widetilde{\mathcal A}\nabla U  )\Big)\psi dtd\theta dx=\int_0^1\int_{\Omega}\int^T_0\nabla\cdot\widetilde {\mathcal C}\psi dtd\theta dx \Big) 
\end{equation}
which is the weak formulation of
\begin{equation}\label{b1}
\frac{\partial U}{\partial \theta} -\nabla\cdot(\widetilde {\mathcal A}\nabla U) =\nabla\cdot \widetilde{\mathcal C}\,\,\text{on}\,\,[0,T)\times\mathbb R\times \Omega
\end{equation}
Using again Green formula in (\ref{vie0}) and (\ref{vie1}), we get finally
\begin{equation}
\frac{\partial U}{\partial \theta}+U=g\,\,\text{on}\,\,[0,T)\times\mathbb R\times \partial\Omega.
\end{equation}
Let us characterize the homogenized equation for $\widetilde{\mathcal A}$ and 
$\widetilde{\mathcal C}.$ Multiplying the first term of (\ref{homoeq}-\ref{homoeq2}) by $\psi^\epsilon$ and integrating over $\Omega$ we get
$$
\int_{\Omega}\int_0^T \widetilde{\mathcal A}_\epsilon\psi^\epsilon dtdx=\int_{\Omega}\int_0^T a(1-b\epsilon \mathcal M(t,\theta,x)g_a(|\mathcal U(t,\theta,x)|)\psi^\epsilon dtdx.
$$
Then passing to the limit as $\epsilon\rightarrow 0,$ 
$$
\int_{\Omega}\int_0^T\int_0^1ag_a(|\mathcal U(t,\theta,x)|)\psi dt dx=\int_{\Omega}\int_0^T\int_0^1\mathcal A\psi d\theta dt dx.
$$
Multiplying also the second term of (\ref{homoeq1})-(\ref{homoeq3}) by $\psi^\epsilon$ and integrating over $\Omega$ we get
$$
\int_{\Omega}\int_0^T\widetilde{\mathcal C}_\epsilon\psi^\epsilon dtdx=\int_{\Omega}\int_0^T c(1-b\epsilon \mathcal M(t,\theta,x))g_c(|\mathcal U(t,\theta,x)|)\frac{\mathcal U(t,\theta,x)}{|\mathcal U(t,\theta,x)|}\psi^\epsilon dtdx.
$$
By letting $\epsilon\rightarrow 0,$ we obtain
$$
\int_{\Omega}\int_0^T\int_0^1cg_c(|\mathcal U(t,\theta,x)|)\frac{\mathcal U(t,\theta,x)}{|\mathcal U(t,\theta,x)|}\psi dt dx=
\int_{\Omega}\int_0^T\int_0^1\mathcal C\psi d\theta dt dx.$$
It follows that
$$\mathcal A=ag_a(|\mathcal U(t,\theta,x)|)\,\,\text{and}\,\, \mathcal C=cg_c(|\mathcal U(t,\theta,x)|)\frac{\mathcal U(t,\theta,x)}{|\mathcal U(t,\theta,x)|}.
$$
\end{proof}
\noindent We have also the following theorem.
Let $\Theta$ and $\Theta_{thr}$ defined by
\begin{equation}
\Theta=[0,T]\times\{\theta\in\mathbb R: \,\widetilde{\mathcal A}(\cdot,\theta,\cdot)=0\}\times \Omega,
\end{equation}
and 
\begin{equation}
\Theta_{thr}=\{(t,\theta,x)\in [0,T)\times\mathbb R\times\Omega\,\,\text{such that}\,\,: \widetilde{\mathcal A}(t,\theta,x)<\tilde G_{thr}\}.
\end{equation}
\begin{theorem}\label{th3} 
Under assumptions  (\ref{hyp1}), (\ref{hupp1}), (\ref{2.6}), (\ref{2.7}), (\ref{coef3}) and (\ref{coef4}) and  for any $T,$ not depending on $\epsilon,$  the sequence $(z^{\epsilon})$ of solutions to  
(\ref{long1}), with coefficients given by (\ref{homoeq})-(\ref{homoeq3}), Two-Scale converges to the profile
$U\in L^{\infty}([0,T],L^{\infty}_\#(\R,L^2(\Omega)))$ solution to
\begin{equation}
\label{eqlim3} \left\{\begin{array}{ccc}
-\nabla\cdot(\widetilde{\mathcal{A}}\nabla U)=\nabla \cdot\widetilde{\mathcal{C}}\,\,\text{in}\,\,(0,T)\times\mathbb R\times \Omega\\
\frac{\partial U}{\partial \theta}=0\,\,\text{on}\,\,\Theta_{thr}\\
\frac{\partial U}{\partial n}+U= g\,\,\text{on}\,\,(0,T)\times\mathbb R\times \partial\Omega\end{array}\right.
\end{equation} where $\widetilde{\mathcal{A}}$ and $\widetilde{\mathcal{C}}$ are given by (\ref{coef4}).
\end{theorem}
%
\begin{proof}
Multiplying (\ref{long1}) by $\psi^\epsilon(t,x)=\psi(t,\frac t\epsilon,x)$ regular with compact support in $[0,T)\times\Omega$ periodic with period 1 in $\theta$ and integrating  over $[0,T)\times\Omega$ we get :
 \begin{equation}
 -\int_{\Omega} z_0(x)\psi(0,0,x) dx+\int_{\Omega}\int^T_0\frac{\partial\psi^\epsilon }{\partial t} z^\epsilon dtdx+\frac{1}{\epsilon^2}\int_{\Omega}\int^T_0 \mathcal A^{\epsilon}\nabla z^{\epsilon}\nabla\psi^\epsilon dtdx$$
 $$-\frac{1}{\epsilon^2}\int_0^T\int_{\partial\Omega} \mathcal A^\epsilon\frac{\partial z^\epsilon}{\partial n}\psi^\epsilon d\sigma=-\frac{1}{\epsilon^2} \int_{\Omega}\int^T_0\mathcal C^\epsilon\cdot\nabla\psi^\epsilon dtdx+\frac{1}{\epsilon^2}\int_0^T\int_{\partial\Omega} \mathcal C^\epsilon\psi^\epsilon.nd\sigma.
\end{equation}
But $\frac{\partial \psi^\epsilon}{\partial t
}$ writes
\begin{equation}\label{deriv1}
\frac{\partial \psi^\epsilon}{\partial t
}=\Big(\frac{\partial\psi}{\partial t}\Big)^\epsilon+
\frac{1}{\epsilon}\Big(\frac{\partial\psi}{\partial \theta}\Big)^\epsilon,
\end{equation}
where
\begin{equation}
\Big(\frac{\partial\psi}{\partial t}\Big)^\epsilon(t,x)=\frac{\partial\psi}{\partial t}(t,\frac{t}{\epsilon},x),\,\,
\,\,\text{and}\,\,\Big(\frac{\partial\psi}{\partial \theta}\Big)^\epsilon(t,x)=\frac{\partial\psi}{\partial \theta}(t,\frac t\epsilon,x)
\end{equation}
Thus,  because of (\ref{deriv1}) and hypothesis (\ref{2.7}) we get
\begin{equation}\label{deriv2}
 \int_{\Omega}\int^T_0 z^\epsilon\Big(\Big(\frac{\partial\psi}{\partial t}\Big)^\epsilon+
\frac{1}{\epsilon}\Big(\frac{\partial\psi}{\partial \theta}\Big)^\epsilon\Big) dtdx+\frac{1}{\epsilon^2}\int_{\Omega}\int^T_0\mathcal A^{\epsilon}\nabla z^\epsilon\,\nabla \psi^\epsilon dtdx $$
$$=-\frac{1}{\epsilon^2} \int_{\Omega}\int^T_0\mathcal C^\epsilon\cdot\nabla\psi^\epsilon dtdx+\int_{\Omega} z_0(x)\psi(0,0,x) dx.
\end{equation}
Then, multiplying (\ref{deriv2}) by $\epsilon^2,$ we get
\begin{equation}
\epsilon^2 \int_{\Omega}\int^T_0 z^\epsilon\Big(\frac{\partial\psi}{\partial t}\Big)^\epsilon dtdx+
\epsilon\int_{\Omega}\int^T_0\Big(\frac{\partial\psi}{\partial \theta}\Big)^\epsilon z^\epsilon dtdx+ \int_{\Omega}\int^T_0\nabla\cdot(\mathcal A^{\epsilon}\nabla \psi^\epsilon) z^\epsilon dtdx $$$$
=- \int_{\Omega}\int^T_0\mathcal C^\epsilon\cdot\nabla\psi^\epsilon dtdx-\epsilon^2\int_{\Omega} z_0(x)\psi(0,0,x) dx.
\end{equation}
As $\psi^\epsilon$ is regular with compact support on $[0,T)\times\Omega,$ and $\mathcal A^\epsilon$ is a regular function, the functions
$\Big(\frac{\partial\psi}{\partial t}\Big)^\epsilon,\,\,\Big(\frac{\partial\psi}{\partial \tau}\Big)^\epsilon,\,\,
\Big(\frac{\partial\psi}{\partial \theta}\Big)^\epsilon,
\,\,\nabla\cdot(\mathcal A^{\epsilon}\nabla \psi^\epsilon)\Big)$ and $\nabla\psi^\epsilon$ can be considered as test functions. Then  using  two-scale convergence as $\epsilon\rightarrow 0,$ we get
\begin{equation}\label{vie}
\int_0^1\int_{\Omega}\int^T_0\nabla\cdot(\widetilde{\mathcal A}\nabla \psi) U dtd\theta dx=- \int_0^1\int_{\Omega}\int^T_0\widetilde{\mathcal C}\cdot\nabla\psi dtd\theta dx.
\end{equation}
Using Green Formula, we get
\begin{equation}
-\int_{\Omega}\int_0^1\int^T_0\nabla\cdot(\widetilde{\mathcal A}\nabla U)\psi dtd\theta dx=\int_0^1\int_{\Omega}\int^T_0\nabla\cdot\widetilde {\mathcal C}\psi dtd\theta dx  
\end{equation}
which is the weak formulation of
\begin{equation}
-\nabla\cdot(\widetilde {\mathcal A}\nabla U) =\nabla\cdot \widetilde{\mathcal C}\\
\end{equation}
Taking test function $\psi^\epsilon(t,x)=\psi(\frac t\epsilon)$ depending only on $\theta,$ and regular with compact support on $\Theta_{thr}$ we get
\begin{equation}
\int_{\Omega}\int^T_0\frac1\epsilon(\frac{\partial \psi}{\partial\theta})^\epsilon z^\epsilon dt\,dx+\frac{1}{\epsilon^2}\int_{\Omega}\int^T_0z^\epsilon    \widetilde{\mathcal A}_\epsilon(\Delta\psi)^\epsilon dt\,dx= \int_\Omega z_0(x)\psi(0,0,x)dx 
\end{equation}
As $\psi^\epsilon$ depends only on $\theta,\,\,\Delta\psi^\epsilon=0,$ then we have
\begin{equation}
\int_{\Omega}\int^T_0(\frac{\partial \psi}{\partial\theta})^\epsilon z^\epsilon dt\,dx= \epsilon\int_\Omega z_0(x)\psi(0,0,x)dx 
\end{equation}
Using Two-scale convergence and passing to the limit, we get
$$\int_{\Omega}\int^T_0\int_0^1\frac{\partial \Psi}{\partial \theta} Ud\theta\,dt\,dx=0,$$
thus 
\begin{equation}
\frac{\partial U}{\partial\theta}=0 \,\,\text{on}\,\,\Theta_{thr}.
\end{equation}
\end{proof}
\subsection{Corrector results}
With Theorem \,\ref{th2} and Theorem\,\ref{th3} in hands, we are going to characterize the first order corrector $U^1$ which prove the best 
approximation by homogenization or two-scale convergence like $z^\epsilon(t,x)\sim U(t,\frac t\epsilon,x)+\epsilon U^1(t,\frac t\epsilon,x)$. To obtain this approximation, 
we look for the equation satisfied by $\frac{z^\epsilon(t, x)-U^\epsilon(t,x)}{\epsilon},$  with \\$U^\epsilon(t,x)=U(t,\frac t\epsilon, x).$ 
Then, we have the following theorem valid for short term model. The corrector result in the mean term is also obtained in the same way.
\begin{theorem}\label{th4}
Under assumptions (\ref{hyp1}), (\ref{hupp1}), (\ref{2.6}), (\ref{2.7}), (\ref{coef1}) and (\ref{coef2})
 considering $z^\epsilon$ the
solution to (\ref{short1}) with coefficients given by (\ref{homoeq})-(\ref{homoeq3}) and 
$U^\epsilon (t, x) = U(t,\frac{t}{\epsilon},x)$ where $U$ is the solution to (\ref{eqlim2}), for any T not
depending on $\epsilon$, the following estimate holds for $z^{\epsilon}-U^{\epsilon}$
\begin{equation}\Big\|\frac{z^{\epsilon}-U^{\epsilon}}{\epsilon}\Big\|_{  L^{\infty}([0,T), L^{2}(\Omega))}\leq\alpha,
\end{equation}
in short term where $\alpha$ is a constant not depending on $\epsilon.$
Furthermore, sequence $\frac{z^\epsilon-U^\epsilon}{\epsilon} $ two-scale converges to a profile $U^1\in
L^\infty([0,T],L^\infty_{\#}(\mathbb R, L^2(\Omega)))$ which is the unique solution to
\begin{equation}\left\{\begin{array}{ccc}
\ds\frac{\partial U^1}{\partial \theta}-\nabla\cdot\Big(\widetilde{\mathcal A}\nabla U^1\Big)=\nabla\cdot\widetilde{\mathcal C}_1+
\frac{\partial U}{\partial t}+
\nabla\cdot\big(\widetilde{\mathcal A}_1\nabla U\big)\,\,\text{in}\,\,]0,T[\times\mathbb R\times\Omega\\
\ds\frac{\partial U^1}{\partial n}+U^1=0\,\,\text{on}\,\, ]0,T[\times\mathbb R\times\partial\Omega\\
U^1(0,0,x)=U^1(x)\,\,\text{in}\,\, \Omega.
\end{array}\right.
\end{equation}
\end{theorem}
\begin{proof}
Since the  coefficients   $\mathcal A^\epsilon(t,x)$ and $\mathcal C^\epsilon(t,x)$ of (\ref{short1})
two scale converges to $\widetilde{\mathcal A}(t,\theta ,x)$ and $\widetilde{\mathcal C}(t,\theta ,x),$  
then these coefficients can be set in the short term in the form
\begin{equation}\label{beq11}
\mathcal A^\epsilon(t,x)=\widetilde{\mathcal A}^\epsilon(t,x)+\epsilon\widetilde{\mathcal A}_1^\epsilon(t,x)\,\,\text{and}\,\,\mathcal C^\epsilon(t,x)=\widetilde{\mathcal C}^\epsilon(t,x)+\epsilon\widetilde{\mathcal C}_1^\epsilon(t,x)\end{equation}
where \begin{equation}
\widetilde{\mathcal A}^\epsilon(t,x)=\widetilde{\mathcal A}(t,\frac t\epsilon,x),\,\,\widetilde{\mathcal C}^\epsilon(t,x)=
\widetilde{\mathcal C}(t,\frac t\epsilon,x)
\end{equation}
and
\begin{equation}
\widetilde{\mathcal A}_1^\epsilon(t,x)=\widetilde{\mathcal A}_1(t,\frac t\epsilon,x),\,\,\widetilde{\mathcal C}_1^\epsilon(t,x)
=\widetilde{\mathcal C}_1(t,\frac t\epsilon,x)
\end{equation}
We have also to notice that, under the same assumptions (\ref{eq40}) and (\ref{hyp1}) the coefficients 
\begin{equation} \label{3.23}\widetilde{\mathcal A}, \,\,\widetilde{\mathcal C},\, \,\widetilde{\mathcal A}_1, \,\widetilde{\mathcal C}_1, \,\, \widetilde{\mathcal A}^\epsilon, \,\,
\widetilde{\mathcal C}^\epsilon, \,\,\widetilde{\mathcal A}_1^\epsilon, \,\,\text{and}\,\,\widetilde{\mathcal C}_1^\epsilon\,\,\text{ are regular and bounded}.\end{equation}
Because of (\ref{beq11}), equation (\ref{short1}) becomes
\begin{equation}\label{ba2}\left\{\begin{array}{ccc}
\ds\frac{\partial z^\epsilon}{\partial t}-\frac1\epsilon\nabla\cdot\big(\widetilde{\mathcal A}^\epsilon\nabla z^\epsilon\big)=\frac1\epsilon\nabla\cdot\widetilde{\mathcal C}^\epsilon+\nabla\cdot\big(\widetilde{\mathcal A}_1^\epsilon\nabla z^\epsilon\big)+\nabla\cdot\widetilde{\mathcal C_1}^\epsilon\,\,\text{in}\,\,]0,T[\times\Omega\\
\frac{\partial z^\epsilon}{\partial n}+z^\epsilon= g\,\,\text{on}\,\,]0,T[\times\partial\Omega.
\end{array}\right.
\end{equation}
From (\ref{b1}) and using the fact that 
\begin{equation}
\frac{\partial U^\epsilon}{\partial t}=\Big(\frac{\partial U}{\partial t}\Big)^\epsilon
+\frac1\epsilon\Big(\frac{\partial U}{\partial \theta}\Big)^\epsilon,
\end{equation}
where  $$\Big(\frac{\partial U}{\partial t}\Big)^\epsilon(t,x)=\frac{\partial U}{\partial t}(t,\frac t\epsilon,x)
\,\,\text{and}\,\,\Big(\frac{\partial U}{\partial \theta}\Big)^\epsilon(t,x)=\frac{\partial U}{\partial \theta}
(t,\frac t\epsilon,x)$$
$U^\epsilon$ is solution to
\begin{equation}\label{ba1}\left\{\begin{array}{ccc}
\ds\frac{\partial U^\epsilon}{\partial t}-\frac1\epsilon\nabla\cdot\big(\widetilde{\mathcal A}^\epsilon\nabla U^\epsilon\big)=\frac1\epsilon\nabla\cdot\widetilde{\mathcal C}^\epsilon+
\Big(\frac{\partial U}{\partial t}\Big)^\epsilon\\
\ds\frac{\partial U^\epsilon}{\partial n}+U^\epsilon= g.
\end{array}\right.
\end{equation}
From formulas (\ref{ba2}) and (\ref{ba1}) we deduce that $\frac{z^\epsilon-U^\epsilon}{\epsilon}$ is solution to
\begin{equation}\label{eq5.8}\left\{\begin{array}{ccc}
\ds\frac{\partial \Big(\frac{z^\epsilon-U^\epsilon}{\epsilon} \Big)}{\partial t}-\frac1\epsilon\nabla\cdot\Big((\widetilde{\mathcal A}^\epsilon+\epsilon\widetilde{\mathcal A}_1^\epsilon)\nabla \Big(\frac{z^\epsilon-U^\epsilon}{\epsilon} \Big)\Big)=\frac1\epsilon\Big(\nabla\cdot\widetilde{\mathcal C}_1^\epsilon+
\Big(\frac{\partial U}{\partial t}\Big)^\epsilon+
\nabla\cdot\big(\widetilde{\mathcal A}^\epsilon_1\nabla U^\epsilon\big)\Big)\\
\,\,\text{in}\,\, ]0,T[\times\Omega\\
\ds\frac{\partial \Big(\frac{z^\epsilon-U^\epsilon}{\epsilon}\Big)}{\partial n}+\frac{z^\epsilon-U^\epsilon}{\epsilon}=0\,\,\text{on}\,\, ]0,T[\times\partial\Omega
\end{array}\right.
\end{equation}
 All the coefficients of (\ref{eq5.8}) are regular and bounded, then existence
 of $\Big(\frac{z^\epsilon-U^\epsilon}{\epsilon} \Big)$ is a consequence result of Ladyzenskaja, 
 Sollonnikov and Ural' Ceva \cite{LadSol}. We have to notice that, as the boundary condition of (\ref{eq5.8})  
 is homogeneous, there is no  the boundary term to be considered. Then using the same argument as in the proof of theorem  
 (\ref{th1.0}), we get that $\Big(\frac{z^\epsilon-Z^\epsilon}{\epsilon} \Big)$  solution to (\ref{eq5.8}) is bounded
in $L^2([0,T), L^2(\Omega))$
 and two-scale converge to $U^1$
 \begin{equation}\left\{\begin{array}{ccc}
\ds\frac{\partial U^1}{\partial \theta}-\nabla\cdot\Big(\widetilde{\mathcal A}\nabla U^1\Big)=\nabla\cdot\widetilde{\mathcal C}_1+
\frac{\partial U}{\partial t}+
\nabla\cdot\big(\widetilde{\mathcal A}_1\nabla U\big)\,\,\text{in}\,\, ]0,T[\times\mathbb R\times\Omega\\
\ds\frac{\partial U^1}{\partial n}+U^1=0\,\,\text{on}\,\, ]0,T[\times\mathbb R\times\partial\Omega\\
U^1(0,0,x)=U^1(x)\,\,\text{in}\,\,\Omega.
\end{array}\right.
\end{equation}
 \end{proof}
\textbf{Acknowledgement.} The authors would like to thank the referee for pointing errors or mistakes.
\end{remark}

\end{document}